\documentclass[11pt]{article}%
\usepackage{graphicx}
\usepackage{amsmath}
\usepackage{amsfonts}
\usepackage{amssymb}%
\providecommand{\U}[1]{\protect\rule{.1in}{.1in}}
\newtheorem{theorem}{Theorem}[section]

\newtheorem{corollary}[theorem]{Corollary}

\newtheorem{definition}[theorem]{Definition}

\newtheorem{lemma}[theorem]{Lemma}
\newtheorem{notation}[theorem]{Notation}

\newtheorem{proposition}[theorem]{Proposition}
\newtheorem{remark}[theorem]{Remark}

\newenvironment{proof}[1][Proof]{\textbf{#1.} }{\hfill\rule{0.5em}{0.5em}}
{\catcode`\@=11\global\let\AddToReset=\@addtoreset
\AddToReset{equation}{section}

\AddToReset{theorem}{section}

\begin{document}

\title{Initial trace of solutions of Hamilton-Jacobi parabolic equation with
absorption }
\author{Marie Fran\c{c}oise BIDAUT-VERON
\and Nguyen Anh DAO}
\date{}
\maketitle

\begin{abstract}
Here we study the initial trace problem for the nonnegative solutions of
equation
\[
u_{t}-\Delta u+|\nabla u|^{q}=0
\]
in $Q_{\Omega,T}=\Omega\times\left(  0,T\right)  ,$ where $q>0,$ and
$\Omega=\mathbb{R}^{N},$ or $\Omega$ is a bounded domain of $\mathbb{R}^{N}$
and $u=0$ on $\partial\Omega\times\left(  0,T\right)  .$ We define the trace
at $t=0$ as a Borel measure $(\mathcal{S},u_{0}),$ infinite on a closed set
$\mathcal{S},$ where $u_{0}$ is a Radon measure on $\Omega\backslash
\mathcal{S}.$ We show that the trace is a Radon measure when $q\leqq1.$ We
study the existence for $q\in(1,(N+2)/(N+1)$ and any given $(\mathcal{S}%
,u_{0})$. When $\mathcal{S}$ $=\overline{\omega}\cap\Omega$ ($\omega$
open$\subset$ $\Omega$) existence is valid for $q\leqq2$ when $u_{0}\in
L_{loc}^{1}(\Omega)$, for $q>1$ when $u_{0}=0$. In particular there exists a
self-similar nonradial solution with trace $(\mathbb{R}^{N+},0),$ with a
growth rate of order $\left\vert x\right\vert ^{q/(q-1)}$ as $\left\vert
x\right\vert \rightarrow\infty$ for fixed $t.$ Moreover the solutions with
trace $(\overline{\omega},0)$ in $Q_{\mathbb{R}^{N},T}$ may present a growth
rate of order $t^{-1/(q-1)}$ in $\omega$ and of order $t^{-(2-q)/(q-1)}$ on
$\partial\omega.$


\textbf{A.M.S. Subject Classification }35K15, 35K55, 35B33, 35B65, 35D30

\end{abstract}
\tableofcontents

\section{Introduction\label{sec1}}

Here we consider the solutions of the parabolic Hamilton-Jacobi equation
\begin{equation}
u_{t}-\Delta u+|\nabla u|^{q}=0 \label{un}%
\end{equation}
in $Q_{\Omega,T}=\Omega\times\left(  0,T\right)  ,$ $T\leqq\infty,$ where
$q>0,$ and $\Omega=\mathbb{R}^{N},$ or $\Omega$ is a smooth bounded domain of
$\mathbb{R}^{N}$ and $u=0$ on $\partial\Omega\times\left(  0,T\right)  .$
\medskip

We mainly study the problem of initial trace of the \textit{nonnegative}
solutions. Our main questions are the following: Assuming that $u$ is a
nonnegative solution, what is the behaviour of $u$ as $t$ tends to $0?$ Does
$u$ converges to a Radon measure $u_{0}$ in $\Omega,$ or even to an unbounded
Borel measure in $\Omega$? Conversely, does there exist a solution with such a
measure as initial data, and is it unique in some class?\medskip

In the sequel $\mathcal{M}(\Omega)$ is the set of Radon measures in
$\Omega,\mathcal{M}_{b}(\Omega)$ the subset of bounded measures$,$ and
$\mathcal{M}^{+}(\Omega),\mathcal{M}_{b}^{+}(\Omega)$ are the cones of
nonnegative ones. We say that a nonnegative solution $u$ of (\ref{un})
\textit{has a trace }$u_{0}$ \textit{in} $\mathcal{M}(\Omega)$ if $u(.,t)$
converges to $u_{0}$ in the weak$^{\ast}$ topology of measures:%
\begin{equation}
\lim_{t\rightarrow0}\int_{\Omega}u(.,t)\psi dx=\int_{\Omega}\psi du_{0}%
,\qquad\forall\psi\in C_{c}(\Omega). \label{mea}%
\end{equation}
First recall some known results. The Cauchy problem in $Q_{\mathbb{R}^{N},T}$
\begin{equation}
(P_{\mathbb{R}^{N},T})\left\{
\begin{array}
[c]{l}%
u_{t}-\Delta u+|\nabla u|^{q}=0,\quad\text{in}\hspace{0.05in}Q_{\mathbb{R}%
^{N},T},\\
u(x,0)=u_{0}\quad\text{in}\hspace{0.05in}\mathbb{R}^{N},
\end{array}
\right.  \label{cau}%
\end{equation}
and the Dirichlet problem in a bounded domain
\begin{equation}
(P_{\Omega,T})\left\{
\begin{array}
[c]{l}%
u_{t}-\Delta u+|\nabla u|^{q}=0,\quad\text{in}\hspace{0.05in}Q_{\Omega,T},\\
u=0,\quad\text{on}\hspace{0.05in}\partial\Omega\times(0,T),\\
u(x,0)=u_{0}.
\end{array}
\right.  \label{1.1}%
\end{equation}
have been the object of a rich literature, see among them \cite{CLS},
\cite{AmBA},\cite{BeLa99}, \cite{BeDa}, \cite{BASoWe}, \cite{SoZh},
\cite{BeBALa}, \cite{BiDao1}, \cite{BiDao2}, and references therein. The first
studies of $(P_{\mathbb{R}^{N},T})$ concern the existence of classical
solutions, that means $u\in C^{2,1}(Q_{\mathbb{R}^{N},T}),$ with smooth
initial data: the case $u_{0}\in C_{b}^{2}\left(  \mathbb{R}^{N}\right)  $ and
$q>1,$ was studied in \cite{AmBA}, and extended to any $u_{0}\in C_{b}\left(
\mathbb{R}^{N}\right)  $ and $q>0$ in \cite{GiGuKe}. Then the problem was
studied in a semi-group formulation for rough initial data $u_{0}\in
L^{r}\left(  \mathbb{R}^{N}\right)  ,$ $r\geqq1,$ or $u_{0}\in\mathcal{M}%
_{b}(\mathbb{R}^{N}),$ \cite{BeLa99}, \cite{BASoWe}, \cite{SoZh}, and in the
larger class of weak solutions in \cite{BiDao1}, \cite{BiDao2}.\medskip\ 

A critical value appears when $q>1:$
\[
q_{\ast}=\frac{N+2}{N+1}.
\]
Indeed the problem with initial value $u_{0}=\delta_{0},$ Dirac mass at $0$
has a weak solution if and only if $q<q_{\ast},$ see \cite{BeLa99},
\cite{BiDao1}. In the same range the problem has a unique very singular
solution (in short V.S.S.) $Y_{\left\{  0\right\}  }$, such that
\[
\lim_{t\rightarrow0}\int_{\left\vert x\right\vert \geqq r}Y_{\left\{
0\right\}  }(.,t)dx=0,\qquad\lim_{t\rightarrow0}\int_{\left\vert x\right\vert
<r}Y_{\left\{  0\right\}  }(.,t)dx=\infty,\qquad\forall r>0,
\]
see \cite{QW}, \cite{BeLa01}, \cite{BeKoLa}, \cite{BiDao1}. It is radial and
self-similar: $Y_{\left\{  0\right\}  }(x,t)=t^{-a/2}F(\left\vert x\right\vert
/\sqrt{t}),$ with
\begin{equation}
F\in C(\left[  0,\infty\right)  ),F(0)>0,F^{\prime}(0)=0,\text{\quad\ }%
\lim_{\left\vert \eta\right\vert \rightarrow\infty}e^{\frac{\eta^{2}}{4}%
}\left\vert \eta\right\vert ^{N-a}F(\eta)=C>0, \label{selfs}%
\end{equation}
where
\begin{equation}
a=\frac{2-q}{q-1}. \label{vala}%
\end{equation}
It is clear that $Y_{\left\{  0\right\}  }$ does not admit a trace as a Radon
measure. Otherwise, for any $q>1,$ the Dirichlet problem $(P_{\Omega,T})$
admits a solution $U$ such that lim$_{t\rightarrow0}U(x,t)=\infty$ uniformly
on the compact sets of $\Omega,$ see \cite{CLS}. Thus we are lead to define an
extended notion of trace.\medskip

The problem has been considered in \cite{BrFr}, \cite{MaVe} for the
semi-linear equation%
\begin{equation}
u_{t}-\Delta u+u^{q}=0, \label{uq}%
\end{equation}
with $q>1.$ Here another critical value $(N+2)/N$ is involved: there exist
solutions with initial value $\delta_{0}$ if and only if $q<(N+2)/N,$ see
\cite{BrFr}, and then there exists a V.S.S., see \cite{BrPeTe}, \cite{KaPe}.
In \cite{MaVe} a precise description of the initial trace is given: any
nonnegative solution admits a trace as an outer regular Borel measure
$\mathcal{U}_{0}$ in $\Omega$. Moreover if $q<(N+2)/N$, the problem is
\textit{well posed} in this set of measures in $\mathbb{R}^{N}.$ The result of
uniqueness lies on the monotony of the function $u\mapsto u^{q}.$ If
$q\geqq(N+2)/N,$ necessary and sufficient conditions are given for existence,
the problem admits a maximal solution, but uniqueness fails. Equation
(\ref{uq}) admits a particular solution $((q-1)t)^{-1/(q-1)}$, which governs
the upper estimates. Notice that the V.S.S. has precisely a behaviour in
$t^{-1/(q-1)}$ at $x=0,$ as $t\rightarrow0.$ \medskip

Here we extend some of these results to equation (\ref{un}). Compared to
problem (\ref{uq}), new difficulties appear:\medskip

1) The first one concerns the a priori estimates. The equation (\ref{un}) has
no particular solution depending only on $t.$ Note also that the sum of two
supersolutions is not in general a supersolution. In \cite{CLS} a universal
upper estimate of the solutions $u,$ of order $t^{-1/(q-1)},$ is proved for
the Dirichlet problem. For the Cauchy problem, universal estimates of the
gradient have been obtained for classical solutions with smooth data $u_{0},$
see \cite{BeLa99}, and \cite{SoZh}. They are improved in \cite{Bi}, where
estimates of $u$ of order $t^{-1/(q-1)}$ are obtained, see Theorem \ref{fund}
below, and it is one of the key points in the sequel.\thinspace\medskip

2) The second one comes from the fact that singular solutions may present two
different levels of singularity as $t\rightarrow0.$ Notice that the V.S.S.
$Y_{\left\{  0\right\}  }$ has a behaviour of order $t^{-a/2}\ll$
$t^{-1/(q-1)}.$\medskip

3) The last one is due to the lack of monotony of the absorption term $|\nabla
u|^{q}$. Thus many uniqueness problems are still open.\medskip

We first recall in Section \ref{1} the notions of solutions, and precise the a
priori upper and lower estimates, for the Cauchy problem or the Dirichlet
problem. In Section \ref{2} we describe the initial trace for $q>1:$

\begin{theorem}
\label{trace} ~Let $q>1.$ Let $u$ be any nonnegative weak solution of
(\ref{un}) in any domain $\Omega.$ Then there exist a set $\mathcal{S}%
\subset\Omega$ such that $\mathcal{R}=\Omega\backslash\mathcal{S}$ is open,
and a measure $u_{0}\in\mathcal{M}^{+}(\mathcal{R})$, such that

$\bullet$ For any $\psi\in C_{c}^{0}(\mathcal{R}),$
\begin{equation}
\lim_{t\rightarrow0}\int_{\mathcal{R}}u(.,t)\psi=\int_{\mathcal{R}}\psi
du_{0}. \label{ccc}%
\end{equation}
\medskip

$\bullet$ For any $x_{0}\in\mathcal{S}$ and any $\varepsilon>0$
\begin{equation}
\lim_{t\rightarrow0}\int_{B\mathcal{(}x_{0},\varepsilon)\cap\Omega
}u(.,t)dx=\infty. \label{ddd}%
\end{equation}

\end{theorem}

The outer regular Borel measure $\mathcal{U}_{0}$ on $\Omega$ associated to
the couple $(\mathcal{S},u_{0})$ defined by
\[
\mathcal{U}_{0}(E)=\left\{
\begin{array}
[c]{c}%
\int_{E}du_{0}\qquad\text{if }E\subset\mathcal{R},\\
\infty\qquad\quad\text{if }E\cap\mathcal{S}\neq\emptyset,
\end{array}
\right.
\]
is called the \textit{initial trace} of $u.$ The set $\mathcal{S}$ is called
the set of \textit{singular points} of $\mathcal{U}_{0}$ and $\mathcal{R}$
called the set of \textit{regular points}, and $u_{0}$\ the \textit{regular
part} of $\mathcal{U}_{0}.$

As $t\rightarrow$ $0$, we give lower estimates of the solutions on
$\mathcal{S}$ of two types: of type $t^{-1/(q-1)}$ on $\overset{\circ
}{\mathcal{S}}$ (if it is nonempty) and of type $t^{-a/2}$ on $\mathcal{S}$
(if $q<q_{\ast}).$ Moreover we describe more precisely the trace for equation
(\ref{un}) in $Q_{\mathbb{R}^{N},T},$ thanks to a recent result of \cite{Bi}
(see Theorem \ref{fund}):

\begin{theorem}
\label{alpha}Let $\mathcal{S}$ be closed set in $\mathbb{R}^{N},$
$\mathcal{S}\neq\mathbb{R}^{N},$ and $u_{0}\in\mathcal{M}^{+}\left(
\mathbb{R}^{N}\backslash\mathcal{S}\right)  .$ Let $u$ be any nonnegative
classical solution of (\ref{un}) in $Q_{\mathbb{R}^{N},T}$ (any weak solution
if $q\leqq2$), \textbf{ }with initial trace $(\mathcal{S},u_{0}).$

Then there exists a measure $\gamma\in\mathcal{M}^{+}(\mathbb{R}^{N}),$
concentrated on $\mathcal{S},$ such that $t^{1/(q-1)}u$ converges weak
$^{\ast}$ to $\gamma$ as $t\rightarrow0.$ And $\gamma\in L_{loc}^{\infty
}(\mathbb{R}^{N});$ in particular if $\left\vert \mathcal{S}\right\vert =0,$
then $\gamma=0;$ if $\mathcal{S}$ is compact, then $\gamma\in L^{\infty
}(\mathbb{R}^{N}).$\medskip
\end{theorem}

In Section \ref{3} we study the existence and the behaviour of solutions with
trace $(\overline{\omega}\cap\Omega,0),$ where $\omega$ is a smooth open
subset of $\Omega$. We construct new solutions of (\ref{un}) in $Q_{\mathbb{R}%
^{N},T}$, in particular the following one:

\begin{theorem}
\label{szero}Let $q>1$, $q^{\prime}=q/(q-1),$ and $\mathbb{R}^{N+}%
=\mathbb{R}^{+}\mathbb{\times R}^{N-1}.$ There exists a nonradial self-similar
solution of (\ref{un}) in $Q_{\mathbb{R}^{N},T},$ with trace $(\overline
{\mathbb{R}^{N+}},0),$ only depending on $x_{1}:$ $U(x,t)=t^{-a/2}%
f(t^{-1/2}x_{1})$, where
\[
\lim_{\eta\rightarrow\infty}\eta^{-q^{\prime}}f(\eta)=c_{q}=(q^{\prime
})^{-q^{\prime}}(\frac{1}{q-1})^{\frac{1}{q-1}},\qquad\lim_{\eta
\rightarrow-\infty}e^{\frac{\eta^{2}}{4}}(-\eta)^{-\frac{3-2q}{q-1}}%
f(\eta)=C>0.
\]
Thus as $t\rightarrow0,$ $U(x,t)$ behaves like $t^{-1/(q-1)}$ for fixed
$x\in\mathbb{R}^{N+},$ and $U(x,t)=f(0)t^{-a/2}$ for $x\in$ $\partial
\mathbb{R}^{N+}.$ And for fixed $t>0,$ $U(x,t)$ is unbounded: it behaves like
$x_{1}^{q^{\prime}}$ as $x_{1}\rightarrow\infty.$\medskip
\end{theorem}

By using $U$ as a barrier, we can estimate precisely the two growth rates of
the solutions in $Q_{\mathbb{R}^{N},T}$ with trace $(\overline{\omega},0),$ on
$\omega$ and on $\partial\omega,$ for any $q>1,$ see Proposition \ref{prec}.
\medskip

In Section \ref{4} we show the existence of solutions with initial trace
$(\mathcal{S},u_{0}),$ when $\mathcal{S}=\overline{\omega}\cap\Omega$ and
$\omega\subset\Omega$ is open, and $u_{0}$ is a measure on $\Omega
\backslash\overline{\omega},$ which can be unbounded, extending the results of
\cite[Theorem 1.4]{Bi} relative to the case of a trace $(0,u_{0})$:

\begin{theorem}
\label{openex} Assume that $\Omega=\mathbb{R}^{N}$ (resp. $\Omega$ is
bounded). Let $\omega$ be a smooth open subset of $\Omega,$ such that
$\mathcal{R}=\Omega\backslash\overline{\omega}$ is nonempty, and let
$\mathcal{S}=\overline{\omega}\cap\Omega$. Let $u_{0}\in\mathcal{M}^{+}\left(
\mathcal{R}\right)  $. We suppose that either $1<q<q_{\ast}$, or $q_{\ast
}\leqq q\leqq2$ and $u_{0}\in L_{loc}^{1}\left(  \mathcal{R}\right)  ,$ or
$q>2$ and $u_{0}\in L_{loc}^{1}\left(  \mathcal{R}\right)  $ is limit of a
nondecreasing sequence of continuous functions.

Then there exists a weak solution $u$ of (\ref{un}) in $Q_{\mathbb{R}^{N},T}$
(resp. a weak solution of $(D_{\Omega,T})$) such that $u$ admits
$(\mathcal{S},u_{0})$ as initial trace. Moreover as $t\rightarrow0,$ $u(.,t)$
converges to $\infty$ uniformly on any compact in $\omega,$ and uniformly on
$\overline{\omega}\cap\Omega$ if $q<q_{\ast}.$ \medskip
\end{theorem}

In the subcritical case $q<q_{\ast}$ we study the existence of solutions with
trace $(\mathcal{S},u_{0})$ for any closed set $\mathcal{S}$ in $\Omega.$ Our
main result is the following:

\begin{theorem}
\label{ess}Let $1<q<q_{\ast},$ and $\Omega=\mathbb{R}^{N}$ (resp. $\Omega$ is
bounded). Let $\mathcal{S}$ be a closed set in $\mathbb{R}^{N},$ such that
$\mathcal{R}=\mathbb{R}^{N}\backslash\mathcal{S}$ is nonempty. Let $u_{0}%
\in\mathcal{M}^{+}\left(  \mathcal{R}\right)  $.\medskip\ 

(i) Then there exists a minimal solution $u$ of (\ref{un}) with initial trace
$(\mathcal{S},u_{0})\medskip$

(ii) If $\mathcal{S}$ is compact in $\Omega$ and $u_{0}\in\mathcal{M}_{b}%
^{+}\left(  \Omega\right)  $ with support in $\mathcal{R}\cup\overline{\Omega
},$ then there exists a maximal solution (resp. a maximal solution such that
$u(.,t)$ converges weakly to $u_{0}$ in $\mathcal{R}$ as $t\rightarrow0).$
\medskip
\end{theorem}

In Section \ref{5} we study equation (\ref{un}) for $0<q\leqq1,$ with more
generally \textit{signed} solutions, and the initial trace of the nonnegative
ones. We first show the local regularity of the signed solutions, see Theorem
\ref{gul2}. We prove a uniqueness result for the Dirichlet problem, extending
to any $0<q\leqq1$ the results of \cite{BeDa}, relative to the case
$0<q<2/(N+1):$

\begin{theorem}
\label{exun}Let $\Omega$ be bounded, $0<q\leqq1,$ and $u_{0}\in\mathcal{M}%
_{b}(\Omega).$ Then there exists a unique weak (signed) solution $u$ of
problem $(P_{\Omega,T})$ with initial data $u_{0}.$ Let $u_{0},v_{0}%
\in\mathcal{M}_{b}(\Omega)$ such that $u_{0}\leqq v_{0}.$ Then $u\leqq v.$ In
particular if $u_{0}\geqq0,$ then $u\geqq0.$ If $u_{0}\leqq0,$ then $u\leqq
0.$\medskip
\end{theorem}

Finally we show that any nonnegative solution admits a trace as a Radon measure:

\begin{theorem}
\label{infun}Let $0<q\leqq1.$ Let $u$ be any nonnegative weak solution of
(\ref{un}) in any domain $\Omega.$ Then $u$ admits a trace $u_{0}$ in
$\mathcal{M}^{+}(\Omega).$
\end{theorem}

\section{First properties of the solutions\label{1}}

We set $Q_{\Omega,s,\tau}=\Omega\times\left(  s,\tau\right)  ,$ for any
$0\leqq s<\tau\leqq\infty,$ thus $Q_{\Omega,T}=Q_{\Omega,0,T}.$ We denote by
$C(\Omega)$ the set of continuous functions in $\Omega,$ and $C_{b}%
(\Omega)=C(\Omega)\cap L^{\infty}(\Omega),$ $C_{c}(\Omega)=\left\{  \varphi\in
C(\Omega):\text{supp}\varphi\subset\subset\Omega\right\}  $, and
$C_{0}(\overline{\Omega})=\left\{  \varphi\in C(\overline{\Omega}%
):\varphi=0\text{ on }\partial\Omega\right\}  .$

\begin{notation}
Let $\Omega=\mathbb{R}^{N}$ or $\Omega$ bounded, and $\Sigma\subset\Omega.$
For any $\delta>0,$ we set
\begin{equation}
\Sigma_{\delta}^{ext}=\left\{  x\in\Omega:d(x,\Sigma)\leqq\delta\right\}
,\qquad\Sigma_{\delta}^{int}=\left\{  x\in\Sigma:d(x,\Omega\backslash
\Sigma)>\delta\right\}  . \label{sigma}%
\end{equation}

\end{notation}

\subsection{Weak solutions and regularity\label{diff}}

\begin{definition}
\label{defw}Let $q>0$ and $\Omega$ be any domain of $\mathbb{R}^{N}.$ We say
that a function $u$ is a \textbf{weak solution} of equation of (\ref{un}) in
$Q_{\Omega,T},$ if $u\in C((0,T);L_{loc}^{1}(Q_{\Omega,T}))\cap L_{loc}%
^{1}((0,T);W_{loc}^{1,1}\left(  \Omega\right)  ),$ $|\nabla u|^{q}\in
L_{loc}^{1}(Q_{\Omega,T}),$ and $u$ satisfies (\ref{un}) in the distribution
sense:
\begin{equation}
\int_{0}^{T}\int_{\Omega}(-u\varphi_{t}-u\Delta\varphi+|\nabla u|^{q}%
\varphi)dxdt=0,\quad\forall\varphi\in\mathcal{D}(Q_{\Omega,T}). \label{for}%
\end{equation}
We say that $u$ is a \textbf{classical} solution of (\ref{un}) in
$Q_{\Omega,T}$ if $u\in C^{2,1}(Q_{\Omega,T})$ and satisfies(\ref{un}) everywhere.

For $u_{0}\in\mathcal{M}^{+}(\mathbb{R}^{N}),$ we say that $u$ is a weak
solution of $(P_{\mathbb{R}^{N},T})$ if u is a weak solution of (\ref{un})
with trace $u_{0}.$
\end{definition}

\begin{remark}
\label{subreg} (i)If $u$ is any nonnegative\textbf{ }function such that $u\in
L_{loc}^{1}(Q_{\Omega,T}),$ and $|\nabla u|^{q}\in L_{loc}^{1}(Q_{\Omega,T}),$
and satisfies (\ref{for}), then $u$ is a weak solution of (\ref{un}). Indeed,
since u is subcaloric, there holds $u\in L_{loc}^{\infty}(Q_{\Omega
,T})),\left\vert \nabla u\right\vert \in L_{loc}^{2}(Q_{\Omega,T})),$ and
$u\in C((0,T);L_{loc}^{\rho}(Q_{\Omega,T})),$ for any $\rho\geqq1,$ see
\cite[Lemma 2.4]{BiDao1} for $q>1;$ the proof is still valid for any $q>0,$
since it only uses the fact that $u$ is subcaloric.

(ii) The weak solutions of $(P_{\mathbb{R}^{N},T})$ are called weak \textbf{
}$\mathcal{M}_{loc}$\textbf{ }solutions in \cite{BiDao2}.
\end{remark}

\begin{definition}
Let $\Omega$ be a smooth bounded domain of $\mathbb{R}^{N}.$ We say that a
function $u$ is a \textbf{weak solution }of\textbf{ }%
\begin{equation}
(D_{\Omega,T})\left\{
\begin{array}
[c]{l}%
u_{t}-\Delta u+|\nabla u|^{q}=0,\quad\text{in}\hspace{0.05in}Q_{\Omega,T},\\
u=0,\quad\text{on}\hspace{0.05in}\partial\Omega\times(0,T),
\end{array}
\right.  \label{diri}%
\end{equation}
if it is a weak solution of (\ref{un}) such that $u\in C((0,T);L^{1}\left(
\Omega\right)  ),$ $u\in L_{loc}^{1}((0,T);W_{0}^{1,1}\left(  \Omega\right)
),$ and $|\nabla u|^{q}\in L_{loc}^{1}((0,T);L^{1}\left(  \Omega\right)  ).$
We say that $u$ is a \textbf{classical} solution of $(D_{\Omega,T})$ if $u\in
C^{1,0}\left(  \overline{\Omega}\times\left(  0,T\right)  \right)  $ and $u$
is a classical solution of (\ref{un}).

For $u_{0}\in\mathcal{M}_{b}(\Omega),$ we say that $u$ is a weak solution of
$(P_{\Omega,T})$ if it is a weak solution of ($D_{\Omega,T})$ such that
$u(.,t)$ converges weakly to $u_{0}$ in $\mathcal{M}_{b}(\Omega):$%
\begin{equation}
\lim_{t\rightarrow0}\int_{\Omega}u(.,t)\psi dx=\int_{\Omega}\psi du_{0}%
,\qquad\forall\psi\in C_{b}(\overline{\Omega}). \label{cos}%
\end{equation}

\end{definition}

Next we recall the regularity of the weak solutions for $q\leqq2,$ see
\cite[Theorem 2.9]{BiDao1}, \cite[Corollary 5.14]{BiDao2}:

\begin{theorem}
\label{gul} Let $1<q\leqq2$.

(i) Let $\Omega$ be any domain in $\mathbb{R}^{N}$, and $u$ be a weak
nonnegative\textbf{ }solution of (\ref{un}) in $Q_{\Omega,T}$. Then $u\in
C_{loc}^{2+\gamma,1+\gamma/2}(Q_{\Omega,T})$ for some $\gamma\in\left(
0,1\right)  .$ Thus for any sequence $(u_{n})$ of nonnegative weak solutions
of (\ref{un}) in $Q_{\Omega,T},$ uniformly locally bounded, one can extract a
subsequence converging in $C_{loc}^{2,1}(Q_{\Omega,T})$ to a weak solution $u$
of (\ref{un}) in $Q_{\Omega,T}.$\medskip

(ii) Let $\Omega$ be bounded, and $u$ be a weak nonnegative\textbf{ }solution
of $(D_{\Omega,T}).$ Then $u\in C^{1,0}\left(  \overline{\Omega}\times\left(
0,T\right)  \right)  $ and $u\in C_{loc}^{2+\gamma,1+\gamma/2}(Q_{\Omega,T})$
for some $\gamma\in\left(  0,1\right)  .$ For any sequence of weak nonnegative
solutions $\left(  u_{n}\right)  $ of $(D_{\Omega,T}),$ one can extract a
subsequence converging in $C_{loc}^{2,1}(Q_{\Omega,T})\cap C_{loc}%
^{1,0}\left(  \overline{\Omega}\times\left(  0,T\right)  \right)  $ to a weak
solution $u$ of $(D_{\Omega,T})$.
\end{theorem}

\subsection{Upper estimates}

We first mention the universal estimates relative to classical solutions of
the Dirichlet problem, see \cite{CLS}, and \cite[Remark 2.8]{BiDao1}:

\begin{theorem}
\label{regdir}Let $q>1,$ and $\Omega$ be any smooth bounded domain. and $u$ be
the classical solution of $(D_{\Omega,T})$ with initial data $u_{0}\in
C^{1,0}\left(  \overline{\Omega}\right)  \cap C_{0}\left(  \overline{\Omega
}\right)  $. Then for any $t\in(0,T),$%
\begin{equation}
\Vert u(.,t)\Vert_{L^{\infty}(\Omega)}\leqq C(1+t^{-\frac{1}{q-1}%
})d(x,\partial\Omega),\qquad\Vert\nabla u(.,t)\Vert_{L^{\infty}(\Omega)}\leqq
D(t), \label{wo}%
\end{equation}
where $C>0$ and $D\in C((0,\infty))$ depend only of $N,q,\Omega$. Thus, for
any sequence $\left(  u_{n}\right)  $ of classical solutions of $(D_{\Omega
,T}),$ one can extract a subsequence converging in $C_{loc}^{2,1}(Q_{\Omega
,T})$ to a classical solution $u$ of $(D_{\Omega,T})$.\medskip
\end{theorem}

Morever some local estimates of classical solutions have been obtained in
\cite{SoZh}, for any $q>1$:

\begin{theorem}
\label{raploc}Let $q>1,$ and $\Omega$ be any domain in $\mathbb{R}^{N},$ and
$u$ be any classical solution of (\ref{un}) in $Q_{\Omega,T}.$ Then for any
ball $B(x_{0},2\eta)\subset\Omega,$ there holds in $Q_{B(x_{0},\eta),T}$%
\begin{equation}
\left\vert \nabla u\right\vert (.,t)\leqq C(t^{-\frac{1}{q}}+\eta^{-1}%
+\eta^{-\frac{1}{q-1}})(1+u(.,t)),\qquad C=C(N,q). \label{soup}%
\end{equation}
Thus, for any sequence of classical solutions $\left(  u_{n}\right)  $ of
(\ref{un}) in $Q_{\Omega,T},$ uniformly bounded in $L_{loc}^{\infty}%
(Q_{\Omega,T}),$ one can extract a subsequence converging in $C_{loc}%
^{2,1}(Q_{\mathbb{R}^{N},T})$ to a classical solution $u$ of (\ref{un}%
).\medskip
\end{theorem}

A local regularizing effect is proved in \cite{Bi}:

\begin{theorem}
\label{local} Let $q>1.$ Let $u$ be any nonnegative weak subsolution of
(\ref{un}) in $Q_{\Omega,T}$, and let $B(x_{0},2\eta)\subset\Omega$ such that
$u$ has a trace $u_{0}\in\mathcal{M}^{+}(B(x_{0},2\eta)).$ Then for any
$\tau<T,$ and any $t\in\left(  0,\tau\right]  ,$
\begin{equation}
\sup_{x\in B(x_{0},\eta/2)}u(x,t)\leqq Ct^{-\frac{N}{2}}(t+\int_{B(x_{0}%
,\eta)}du_{0}),\qquad C=C(N,q,\eta,\tau). \label{locma}%
\end{equation}

\end{theorem}

Concerning the Cauchy problem in $Q_{\mathbb{R}^{N},T}$, global regularizing
effects have been obtained in \cite{BiDao2} for weak solutions with trace
$u_{0}$ in $L^{r}(\mathbb{R}^{N}),r\geqq1,$ or in $\mathcal{M}_{b}%
(\mathbb{R}^{N})$. A universal estimate of the gradient was proved in
\cite{BeLa99} for any classical solution of (\ref{un}) in $Q_{\mathbb{R}%
^{N},\infty}$ such that $u\in C_{b}(\overline{Q_{\mathbb{R}^{N},\infty}})$.
From \cite{Bi}, this estimate is valid without conditions as $\left\vert
x\right\vert \rightarrow\infty,$ implying growth estimates of the function:

\begin{theorem}
\label{fund}Let $q>1.$ Let $u$ be any classical solution, in particular
\textbf{ }any weak solution if $q\leqq2,$ of (\ref{un}) in $Q_{\mathbb{R}%
^{N},T}.$ Then
\begin{equation}
\left\vert \nabla u(.,t)\right\vert ^{q}\leqq\frac{1}{q-1}\frac{u(.,t)}%
{t},\qquad\text{in }Q_{\mathbb{R}^{N},T}. \label{versa}%
\end{equation}
Moreover, if there exists a ball $B(x_{0},2\eta)$ such that $u$ has a trace
$u_{0}\in\mathcal{M}^{+}((B(x_{0},2\eta)),$ then for any $t\in\left(
0,T\right)  ,$%
\begin{equation}
u(x,t)\leqq C(q)t^{-\frac{1}{q-1}}\left\vert x-x_{0}\right\vert ^{q^{\prime}%
}+C(t^{-\frac{1}{q-1}}+t+\int_{B(x_{0},\eta)}du_{0}),\qquad C=C(N,q,\eta).
\label{pluc}%
\end{equation}

\end{theorem}

Finally we recall some well known estimates, useful in the subcritical case,
see \cite[Lemma 3.3]{BaPi}:\medskip

\begin{theorem}
\label{bapi} Let $q>0$ and let $\Omega$ be any domain of $\mathbb{R}^{N}$ and
$u$ be any (signed) weak solution of equation of (\ref{un}) in $Q_{\Omega,T}$
(resp. of $(D_{\Omega,T})).$ Then, $u\in L_{loc}^{1}((0,T);W_{loc}%
^{1,k}(\Omega),$ for any $k\in\left[  1,q_{\ast}\right)  ,$ and for any open
set $\omega\subset\subset\Omega,$ and any $0<s<\tau<T,$
\begin{equation}
\left\Vert u\right\Vert _{L^{k}((s,\tau);W^{1,k}(\omega))}\leqq C(k,\omega
)(\left\Vert u(.,s)\right\Vert _{L^{1}(\omega)}+\left\Vert |\nabla
u|^{q}+|\nabla u|+\left\vert u\right\vert \right\Vert _{L^{1}(Q_{\omega
,s,\tau})}). \label{pi}%
\end{equation}
If $\Omega$ is bounded, any solution $u$ of $(D_{\Omega,T})$ satisfies $u\in
L^{k}((s,\tau);W_{0}^{1,k}(\Omega))$, for any $k\in\left[  1,q_{\ast}\right)
,$ and
\begin{equation}
\left\Vert u\right\Vert _{L^{k}((s,\tau);W_{0}^{1,k}(\Omega))}\leqq
C(k,\Omega)(\left\Vert u(.,s)\right\Vert _{L^{1}(\Omega)}+\left\Vert |\nabla
u|^{q}\right\Vert _{L^{1}(Q_{\Omega,s,\tau})}). \label{pa}%
\end{equation}

\end{theorem}

\subsection{Uniqueness and comparison results}

Next we recall some known results, for the Cauchy problem, see \cite[Theorems
2.1,4.1,4.2 and Remark 2.1 ]{BASoWe},\cite[Theorem 2.3, 4.2, 4.25, Proposition
4.26 ]{BiDao2}, and for the Dirichlet problem, see \cite[Theorems 3.1,
4.2]{Al}, \cite{BeDa}, \cite[Proposition 5.17]{BiDao2}, \cite{Po}.

\begin{theorem}
\label{souc} Let $\Omega=\mathbb{R}^{N}$ (resp. $\Omega$ bounded). (i) Let
$1<q<q_{\ast},$ and $u_{0}\in\mathcal{M}_{b}(\mathbb{R}^{N})$ $($resp.
$u_{0}\in\mathcal{M}_{b}(\Omega)$). Then there exists a unique weak solution
$u$ of (\ref{un}) with trace $u_{0}$ (resp. of $(P_{\Omega,T})).$ If $v_{0}%
\in\mathcal{M}_{b}(\Omega)$ and $u_{0}\leqq v_{0},$ and $v$ is the solution
with trace $v_{0},$ then $u\leqq v.$

(ii) Let $u_{0}\in L^{r}\left(  \Omega\right)  ,$ $1\leqq r\leqq\infty.$ If
$1<q<(N+2r)/(N+r),$ or if $q=2,$ $r<\infty,$ there exists a unique weak
solution $u$ of $(P_{\mathbb{R}^{N},T})$ (resp. $(P_{\Omega,T})$) such that
$u\in C(\left[  0,T\right)  ;L^{r}\left(  \mathbb{R}^{N}\right)  ).$ If
$v_{0}\in L^{r}\left(  \mathbb{R}^{N}\right)  $ and $u_{0}\leqq v_{0},$ then
$u\leqq v.$ If $u_{0}$ is nonnegative, then for any $1<q\leqq2,$ there still
exists a weak nonnegative solution $u$ of $(P_{\mathbb{R}^{N},T})$ (resp.
$(P_{\Omega,T})$) such that $u\in C(\left[  0,T\right)  ;L^{r}\left(
\mathbb{R}^{N}\right)  ).$\medskip
\end{theorem}

\begin{remark}
\label{evi} Let $1\leqq q<q_{\ast},$ and $u_{0}\in\mathcal{M}_{b}%
^{+}(\mathbb{R}^{N})$ and $u$ be the solution of $(P_{\mathbb{R}^{N},T})$ in
$\mathbb{R}^{N},$ and $u^{\Omega}$ be the solution of $(D_{\Omega,T})$ for
bounded $\Omega$ with initial data $u_{0}^{\Omega}=u_{0}\llcorner\Omega,$ then
$u^{\Omega}\leqq u.$\medskip
\end{remark}

We add to the results above a stability property needed below:

\begin{proposition}
\label{cpro}Assume that $1<q<q_{\ast}.$ Let $\Omega=\mathbb{R}^{N}$ (resp.
$\Omega$ be bounded), and $u_{0,n},u_{0}\in\mathcal{M}_{b}^{+}(\Omega)$ such
that $(u_{0,n})$ converge to $u_{0}$ weakly in $\mathcal{M}_{b}(\Omega).$ Let
$u_{n},u$ be the (unique) nonnegative solutions of (\ref{un}) in
$Q_{\mathbb{R}^{N},T}$ (resp. of $(D_{\Omega,T})$) with initial data
$u_{0,n},u_{0}.$ Then $(u_{n})$ converges to $u$ in $C_{loc}^{2,1}%
(Q_{\mathbb{R}^{N},T})$ (resp. in $C_{loc}^{2,1}(Q_{\Omega,T})\cap
C^{1,0}\left(  \overline{\Omega}\times\left(  0,T\right)  \right)  $).\medskip
\end{proposition}

\begin{proof}
(i) From \cite[Theorem 2.2]{BiDao2}, $(u_{n})$ is uniformly locally bounded in
$Q_{\mathbb{R}^{N},T}$ in case $\Omega=\mathbb{R}^{N}.$ From Theorem
\ref{gul}, one can extract a subsequence still denoted $\left(  u_{n}\right)
$ converging in $C_{loc}^{2,1}(Q_{\mathbb{R}^{N},T})$ (resp. $C_{loc}%
^{2,1}(Q_{\Omega,T})\cap C^{1,0}\left(  \overline{\Omega}\times\left(
0,T\right)  \right)  )$ to a classical solution $w$ of (\ref{un}) in
$Q_{\mathbb{R}^{N},T}$ (resp. of $(D_{\Omega,T})$). From uniqueness, we only
have to show that $w(.,t)$ converges weakly in $\mathcal{M}_{b}(\Omega)$ to
$u_{0}.$ In any case, from \cite[Theorem 4.15 and Lemma 5.11]{BiDao2},
$|\nabla u_{n}|^{q}\in L_{loc}^{1}(\left[  0,T\right)  ;L^{1}(\Omega))$ and
\begin{equation}
\int_{\Omega}u_{n}(.,t)dx+\int_{0}^{t}\int_{\Omega}|\nabla u_{n}|^{q}%
dx\leqq\int_{\Omega}du_{0,n}, \label{auto}%
\end{equation}
and $\lim\int_{\Omega}du_{0,n}=\int_{\Omega}du_{0}.$ Therefore $\left(
u_{n}\right)  $ is bounded in $L^{\infty}((0,T),L^{1}(\Omega)),$ and $\left(
|\nabla u_{n}|^{q}\right)  $ is bounded in $L_{loc}^{1}(\left[  0,T\right)
;L^{1}(\Omega)).$ From Theorem \ref{bapi}, for any $k\in\left[  1,q_{\ast
}\right)  ,$ $\left(  u_{n}\right)  $ is bounded in $L^{k}((0,T),W_{loc}%
^{1,k}(\mathbb{R}^{N}))$ (resp. $L^{k}((0,T),W_{0}^{1,k}(\Omega)$)$.$ Then for
any $\tau\in\left(  0,T\right)  ,$ $($ $\left\vert \nabla u_{n}\right\vert
^{q})$ is equi-integrable in $Q_{B_{R},\tau}$ for any $R>0$ (resp. in
$Q_{\Omega,\tau}).$ For any $\xi\in C_{c}^{1}(\mathbb{R}^{N})$ (resp. $\xi\in
C_{b}^{1}(\Omega)),$%
\[
\int_{\Omega}u_{n}(.,t)\xi dx+\int_{0}^{t}\int_{\Omega}(\nabla u_{n}.\nabla
\xi+\left\vert \nabla u_{n}\right\vert ^{q}\xi)dxdt=\int_{\Omega}\xi
du_{0,n},
\]
and we can go to the limit and obtain%
\[
\int_{\Omega}w(.,t)\xi dx+\int_{0}^{t}\int_{\Omega}(\nabla w.\nabla
\xi+\left\vert \nabla u\right\vert ^{q}\xi)dxdt=\int_{\Omega}\xi du_{0},
\]
Then $w$ is a weak solution of\textbf{ }$(P_{\Omega,T})$, unique from Theorem
\ref{souc}, thus $w=u.$\medskip
\end{proof}

\begin{corollary}
\label{cpri} Suppose $1<q<q_{\ast},$ $\Omega$ bounded, and let $v$ $\in
C^{2,1}(Q_{\Omega,T})\cap C^{0}(\overline{\Omega}\times(0,T))$ such that
\[
v_{t}-\Delta v+\left\vert \nabla v\right\vert ^{q}\geqq0,\qquad\text{in
}\mathcal{D}^{\prime}(Q_{\Omega,T}),
\]
and $v_{/\Omega}$ has a trace $u_{0}\in\mathcal{M}_{b}(\Omega).$ Let $w$ be
the solution of $(D_{\Omega,T})$ with trace $u_{0}.$ Then $v\geqq w.$
\end{corollary}

\begin{proof}
Let $\epsilon>0$ and $\left(  \varphi_{n}\right)  $ be a sequence in
$\mathcal{D}^{+}(\Omega)$ with values in $\left[  0,1\right]  ,$ such that
$\varphi_{n}(x)=1$ if $d(x,\partial\Omega)>1/n,$ and $w_{n}^{\epsilon}$ be the
solution of $(D_{\Omega,T})$ with trace $\varphi_{n}v(.,\epsilon)$ at time
$0,$ unique from Theorem \ref{souc}. From \cite[Proposition 2.1]{SoZh},
$v(.,t+\epsilon)\geqq w_{n}^{\epsilon}.$ As $n\rightarrow\infty,$
$(\varphi_{n}v(.,\epsilon))$ converges to $v(.,\epsilon)$ in $L^{1}(\Omega),$
then from Proposition \ref{cpro}, $(w_{n}^{\epsilon})$ converges to the
solution $w^{\epsilon}$ of $(D_{\Omega,T})$ with trace $v(.,\epsilon).$ Then
$v(.,t+\epsilon)\geqq w^{\epsilon}.$ As $\epsilon\rightarrow0,$ $(v(.,\epsilon
))$ converges to $u_{0}$ weakly in $\mathcal{M}_{b}(\Omega),$ thus
$(w^{\epsilon})$ converges to $w,$ thus $v\geqq w.$
\end{proof}

\subsection{The case of zero initial data}

Here we give more informations on the behaviour of the solutions with trace
$0$ on some open set. We show that the solutions are locally uniformly bounded
on this set and converge locally exponentially to $0$ as $t\rightarrow0$,
improving some results of \cite{CLS} for the Dirichlet problem.

\begin{lemma}
\label{bord} Let $\mathcal{F}$ be a closed set in $\mathbb{R}^{N}$,
$\mathcal{F}\neq\mathbb{R}^{N}$ (resp. a compact set in $\Omega$
bounded).\medskip

(i) Let $u$ be a classical solution of (\ref{un}) in $Q_{\mathbb{R}^{N},T}$
(resp. $(D_{\Omega,T})$) such that $u\in C(\left[  0,T\right)  ,C_{b}%
(\mathbb{R}^{N}))$ (resp. $u\in C(\left[  0,T\right)  ;C_{0}(\overline{\Omega
}))$ and supp$u(0)\subset\mathcal{F}.$ Then for any $\delta>0$, (resp. such
that $\delta<d(\mathcal{F},\partial\Omega)/2$)
\begin{equation}
\left\Vert u(.,t)\right\Vert _{L^{\infty}(\Omega\backslash\mathcal{F}_{\delta
}^{ext})}\leqq C(N,q,\delta)t,\qquad\forall t\in\left[  0,T\right)  .
\label{ita}%
\end{equation}
In particular $u(.,t)$ converges uniformly to $0$ on $\Omega\backslash
\mathcal{F}_{\delta}^{ext}$ as $t\rightarrow0.$ Moreover, there exist
$C_{i,\delta}=C_{i,\delta}(N,q,\delta)>0$ $(i=1,2)$, and $\tau_{\delta}>0$
such that
\begin{equation}
\left\Vert u(.,t)\right\Vert _{L^{\infty}(\Omega\backslash\mathcal{F}_{\delta
}^{ext})}\leqq C_{1,\delta}e^{-\frac{C_{2,\delta}}{t}}\text{ on }\left(
0,\tau_{\delta}\right]  . \label{expo}%
\end{equation}
(ii) As a consequence, for any classical solution $w$ of (\ref{un}) in
$Q_{\mathbb{R}^{N},T}$ (resp. $(D_{\Omega,T})$) such that $w(.,t)$ converges
to $\infty$ as $t\rightarrow0,$ uniformly on $\mathcal{F}_{\delta}^{ext},$ for
some $\delta>0,$ there holds $u\leqq w.\medskip$

(iii) If $q\leqq2,$ then (i) still holds for any weak solution $u$ of
(\ref{un}) (resp. of $(D_{\Omega,T})$) with trace $0$ in $\mathcal{M}%
(\mathbb{R}^{N}\backslash\mathcal{F})$ (resp. which converges weakly to $0$ in
$\mathcal{M}_{b}(\Omega\backslash\mathcal{F})$), and (ii) holds if
$\mathcal{F}\subset\subset\Omega.$
\end{lemma}

\begin{proof}
From \cite[Lemma 3.2]{Bi}, for any domain $\Omega$ of $\mathbb{R}^{N},$ if $u$
is any classical solution of (\ref{un}) in $Q_{\Omega,T}$ such that $u\in
C(\Omega\times\left[  0,T\right)  ),$ for any ball $B(x_{0},3\eta
)\subset\Omega,$ and any $t\in\left[  0,T\right)  ,$%
\begin{equation}
\left\Vert u(.,t)\right\Vert _{L^{\infty}(B(x_{0},\eta))}\leqq C(N,q)\eta
^{-q^{\prime}}t+\left\Vert u_{0}\right\Vert _{L^{\infty}(B(x_{0},2\eta))}.
\label{iss}%
\end{equation}

(i) Let $\Omega$ be arbitrary. For any $x_{0}\in\Omega\backslash
\mathcal{F}_{\delta}^{ext},$ taking $\eta=\delta/3$ we deduce (\ref{ita})
follows. Next suppose $\Omega$ bounded and $\mathcal{F}$ compact. Consider a
regular domain $\Omega^{\prime}$ such that $\mathcal{F}_{2\delta}^{ext}%
\subset\Omega^{\prime}\subset\subset\Omega.$ Let $\gamma=d(\overline
{\Omega^{\prime}},\partial\Omega).$ For any $x_{0}\in\overline{\Omega^{\prime
}}\backslash\mathcal{F}_{\delta}^{ext},$ taking $\eta=\min(\delta
/3,\gamma/3),$ we have $B(x_{0},3\eta)\subset\Omega\backslash\mathcal{F}$ thus
we still get (\ref{ita}). As a consequence $u(.,t)\leqq Ct$ in $\overline
{\Omega^{\prime}}\backslash\mathcal{F}_{\delta}^{ext},$ with $C=C(N,q,\delta
,\gamma),$ in particular on $\partial\Omega^{\prime}.$ Following an argument
of \cite[Lemma 4.8]{BiDao1}, the function $z=u-Ct$ solves
\[
z_{t}-\Delta z=-\left\vert \nabla u\right\vert ^{q}-C\text{ in }%
\Omega\backslash\overline{\Omega^{\prime}}%
\]
then $z^{+}$ is subcaloric and $z^{+}=0$ on the parabolic boundary of
$\Omega\backslash\overline{\Omega^{\prime}},$ thus $z^{+}=0.$ Thus
$u(.,t)\leqq Ct$ in $\overline{\Omega}\backslash\mathcal{F}_{\delta}^{ext}.$

Next consider the behaviour for small $t.$ We use a supersolution in
$B_{1}\times\left[  0,\infty\right)  $ given in \cite[Proposition 5.1]{PoZu}.
Let $\alpha\in\left(  0,1/2\right)  ,$ and $d_{\alpha}(x)$ radial: $d_{\alpha
}(x)=d_{\alpha}(\left\vert x\right\vert ),$ with $d_{\alpha}\in C^{2}(\left[
0,1\right)  ),$ $d_{\alpha}(r)=1-r$ for $1-r<\alpha,$ $d_{\alpha}%
(r)=3\alpha/2$ for $1-r>2\alpha,$ $\left\vert \nabla d_{\alpha}\right\vert
\leqq1,$ $\left\vert \Delta d_{\alpha}\right\vert \leqq C(N)d_{\alpha}^{-2}.$
Let
\[
v(x,t)=e^{\frac{1}{d_{\alpha}(x)}-m\frac{d_{\alpha}(x)^{3}}{t}}%
\]
with $m\leqq m(N)$ small enough. Then if $\alpha\leqq\alpha(N)$ small enough,
there exists $\tau(\alpha)>0$ such that $v$ is a supersolution of (\ref{un})
in $B_{1}\times\left(  0,\tau(\alpha)\right]  .$ Then $v(x,t)=C_{1}%
(\alpha)e^{-C_{2}(\alpha)/t}$ in $B_{1/2}\times\left(  0,\tau(\alpha)\right]
.$ And $v$ is infinite \ on $\partial B_{1}\times\left(  0,\tau(\alpha
)\right]  $ and vanishes on $B_{1}\times\left\{  0\right\}  .$ Then by
scaling, for any $x_{0}\in\mathbb{R}^{N}\backslash\mathcal{F}_{\delta}^{ext}$
(resp. $x_{0}\in\overline{\Omega\backslash\mathcal{F}_{\delta}^{ext}}),$ from
the comparison principle in $B(x_{0},\delta)\cap\overline{\Omega}$, we get
\begin{equation}
u(x_{0},t)\leqq\delta^{-a}v(x_{0}/\delta,t/\delta^{2})\leqq C_{1}%
(\alpha)\delta^{-a}e^{-C_{2}(\alpha)\delta^{2}/t} \label{all}%
\end{equation}
and (\ref{expo}) follows.\medskip

(ii) Suppose that $w(.,t)$ converges to $\infty$ as $t\rightarrow0,$ uniformly
on $\mathcal{F}_{\delta}^{ext}.$ Then for any $\epsilon_{0}>0,$ there exists
$\tau_{0}\in\left(  0,T\right)  $ such that $u(.,t)\leqq\epsilon_{0}$ in
$\overline{\Omega}\backslash\mathcal{F}_{\delta}\times\left(  0,\tau
_{0}\right]  .$ Let $\epsilon<\tau_{0}.$ then there exists $\tau_{\epsilon
}<\tau_{0}$ such that for any $\theta\in\left(  0,\tau_{\epsilon}\right)  ,$
$w(.,\theta)\geqq\max_{\overline{\Omega}}u(.,\epsilon)$ in $\mathcal{F}%
_{\delta}$. Then $u(.,t+\epsilon)\leqq w(.,t+\theta)+\epsilon_{0},$ in
$\overline{\Omega}\times\left(  0,\tau_{0}-\epsilon\right]  $ from the
comparison principle. As $\theta\rightarrow0,$ then $\epsilon\rightarrow0,$ we
get $u(.,t)\leqq w(.,t)+\epsilon_{0},$ in $\overline{\Omega}\times\left(
0,\tau_{0}\right]  .$ From the comparison principle, $u(.,t)\leqq
w(.,t)+\epsilon_{0},$ in $\overline{\Omega}\times(0,T).$ As $\epsilon
_{0}\rightarrow0,$ we deduce that $u\leqq w$.\medskip

(iii) Assume $q\leqq2.$ First suppose $\Omega=\mathbb{R}^{N}.$ From
\cite[Proposition 2.17 and Corollary 2.18]{BiDao1}, the extension
$\overline{u}$ of $u$ by $0$ to $(-T,T)$ is a weak solution in $Q_{\mathbb{R}%
^{N}\backslash\mathcal{F},-T,T},$ hence $u\in C^{2,1}(\mathbb{R}^{N}%
\backslash\mathcal{F}\times\left[  0,T\right)  ),$ then $u$ is a classical
solution of (\ref{un}) in $Q_{\mathbb{R}^{N}\backslash\mathcal{F},-T,T}$; thus
(\ref{ita}) and (\ref{expo}) follow. Moreover, if $\mathcal{F}$ is compact,
then $u(.,\epsilon/2)\in C_{b}(\mathbb{R}^{N})$ from (\ref{ita}), then
$u(.,\epsilon)\in C_{b}^{2}(\mathbb{R}^{N})$, thus we still obtain $u\leqq w$
from the comparison principle. Next suppose $\Omega$ bounded and $\mathcal{F}$
compact. Arguing as in \cite[Lemma 4.8]{BiDao1}, we show that $u\in
C^{0}(\overline{\Omega\backslash\mathcal{F}_{\delta}^{ext}}\times\left[
0,T\right)  ),$ and $u(0)=0$ in $\overline{\Omega\backslash\mathcal{F}%
_{\delta}^{ext}}.$ We still get (\ref{ita}) by considering $z$ as above, and
using the Kato inequality, and (\ref{expo}) from the comparison principle.
Moreover we still get $u\leqq w$.
\end{proof}

\section{Existence of initial trace as a Borel measure\label{2}}

Recall a simple trace result of \cite{BiDao1}.

\begin{lemma}
\label{phi}Let $\Omega$ be any domain of $\mathbb{R}^{N},$ and $U\in
C((0,T);L_{loc}^{1}(\Omega))$ be any nonnegative weak solution of equation
\begin{equation}
U_{t}-\Delta U=\Phi\quad\text{in }Q_{\Omega,T}, \label{4.7}%
\end{equation}
with $\Phi\in L_{loc}^{1}(Q_{\Omega,T}),$ $\Phi\geqq-G,$ where $G\in
L_{loc}^{1}(\Omega\times\lbrack0,T))$. Then $U(.,t)$ admits a trace $U_{0}%
\in\mathcal{M}^{+}(\Omega).$ Furthermore, $\Phi\in L_{loc}^{1}([0,T);L_{loc}%
^{1}(\Omega)),$ and for any $\varphi\in C_{c}^{2}(\Omega\times\lbrack0,T))$,
\begin{equation}
-\int_{0}^{T}\int_{\Omega}(U\varphi_{t}+U\Delta\varphi+\Phi\varphi
)dxdt=\int_{\Omega}\varphi(.,0)dU_{0}. \label{tru}%
\end{equation}
If $\Phi$ has a constant sign, then $U\in L_{loc}^{\infty}{(}\left[
0,T\right)  {;L_{loc}^{1}(}\Omega))$ if and only if $\Phi\in L_{loc}%
^{1}([0,T);L_{loc}^{1}(\Omega)).$\medskip
\end{lemma}

As a consequence, we get a characterization of the solutions of (\ref{un}) in
any domain $\Omega$ which have a trace in $\mathcal{M}^{+}(\Omega):$ as in
\cite[Proposition 2.15]{BiDao1} in case $q>1,$ we find:

\begin{proposition}
\label{dic} Let $q>0.$ Let $u$ be any nonnegative weak solution $u$ of
(\ref{un}) in $Q_{\Omega,T}$. Then the following conditions are equivalent:

(i) $u$ has a trace $u_{0}$ in $\mathcal{M}^{+}(\Omega),$

(ii) $u\in L_{loc}^{\infty}{(}\left[  0,T\right)  {;L_{loc}^{1}(}\Omega)),$

(iii) $\left\vert \nabla u\right\vert ^{q}\in L_{loc}^{1}(\Omega\times\left[
0,T\right)  ).$

\noindent And then for any $t\in(0,T),$ and any $\varphi\in C_{c}^{1}%
(\Omega\times\left[  0,T\right)  )$,
\begin{equation}
\int_{\Omega}u(.,t)\varphi dx+\int_{0}^{t}\int_{\Omega}(-u\varphi_{t}+\nabla
u.\nabla\varphi+\left\vert \nabla u\right\vert ^{q}\varphi)dxdt=\int_{\Omega
}\varphi(.,0)du_{0}. \label{hou}%
\end{equation}
And if $q>1,$ for any nonnegative $\zeta,\xi\in C_{c}^{1}(\Omega),$%
\begin{equation}
\int_{\Omega}u(.,t)\zeta dx+\int_{0}^{t}\int_{\Omega}(\nabla u.\nabla
\zeta+\left\vert \nabla u\right\vert ^{q}\zeta)dxdt=\int_{\Omega}\zeta du_{0},
\label{bou}%
\end{equation}%
\begin{equation}
\int_{\Omega}u(.,t)\xi^{q^{\prime}}dx+\frac{1}{2}\int_{0}^{t}\int_{\Omega
}|\nabla u|^{q}\xi^{q^{\prime}}dx\leqq C(q)t\int_{\Omega}|\nabla
\xi|^{q^{\prime}}dx+\int_{\Omega}\xi^{q^{\prime}}du_{0}. \label{zif}%
\end{equation}

\end{proposition}

\begin{proof}
The equivalence and equality (\ref{hou}) hold from Lemma \ref{phi}. Moreoever
for any $0<s<t<T,$
\begin{align*}
&  \int_{\Omega}u(.,t)\xi^{q^{\prime}}dx+\int_{s}^{t}\int_{\Omega}|\nabla
u|^{q}\xi^{q^{\prime}}dx=-q^{\prime}\int_{s}^{t}\int_{\Omega}\xi
^{1/(q-1)}\nabla u.\nabla\xi dx+\int_{\Omega}u(s,.)\xi^{q^{\prime}}dx\\
&  \leqq\frac{1}{2}\int_{s}^{t}\int_{\Omega}|\nabla u|^{q}\xi^{q^{\prime}%
}dx+C(q)t\int_{\Omega}|\nabla\xi|^{q^{\prime}}dx+\int_{\Omega}u(.,s)\xi
^{q^{\prime}}dx,
\end{align*}
hence we obtain (\ref{zif}) as $s\rightarrow0.$\medskip
\end{proof}

\begin{remark}
\label{dac}Note that $u\in L_{loc}^{\infty}{(}\left[  0,T\right)
{;L_{loc}^{1}(}\Omega))$ if and only if $\lim\sup_{t\rightarrow0}\int%
_{B(x_{0},\rho)}u(.,t)dx$ is finite, for any ball $B(x_{0},\rho)\subset
\subset\Omega.$\medskip
\end{remark}

\begin{remark}
\label{norm}If $\Omega$ is bounded, $u_{0}\in\mathcal{M}_{b}^{+}(\Omega)$ and
$u$ is any nonnegative classical solution (resp. weak solution if $q\leqq2)$
of $(P_{\Omega,T}),$ then (\ref{zif}) still holds for any nonnegative $\xi\in
C_{b}^{1}(\Omega).$ Indeed for any $0<s<t<T,$ (\ref{bou}) is replaced by an
inequality%
\[
\int_{\Omega}u(.,t)\zeta dx+\int_{0}^{t}\int_{\Omega}(\nabla u.\nabla\zeta
dx+\left\vert \nabla u\right\vert ^{q}\zeta)dxdt=\int_{0}^{t}\int%
_{\partial\Omega}\frac{\partial u}{\partial\nu}\zeta dsdt+\int_{\Omega
}u(.,s)\zeta dx\leqq\int_{\Omega}u(.,s)\zeta dx,
\]
and (\ref{zif}) follows as above.
\end{remark}

Then we prove the trace Theorem:\medskip

\begin{proof}
[Proof of Theorem \ref{trace}]Let $q>1.$ Let $u$ be any nonnegative weak
solution of (\ref{un}) in $Q_{\Omega,T}.$\medskip

(i) Let $x_{0}\in\Omega.$ Then the following alternative holds (for any
$\tau<T)$:\medskip

$(A1)$ Either there exists a ball $B(x_{0},\rho)\subset\Omega$ such that
$\int_{0}^{\tau}\int_{B(x_{0},\rho)}\left\vert \nabla u\right\vert
^{q}dxdt<\infty.$ Then from Lemma \ref{phi} in $B(x_{0},\rho),$ there exists a
measure $m_{\rho}\in\mathcal{M}^{+}(B(x_{0}$,$\rho))$, such that for any
$\psi\in C_{c}^{0}(B(x_{0},\rho)\mathcal{)}$,%
\begin{equation}
\lim_{t\rightarrow0}\int_{B(x_{0},\rho)}u(.,t)\psi=\int_{B(x_{0},\rho)}\psi
dm_{\rho}, \label{mro}%
\end{equation}

$(A2)$ Or for any ball $B(x_{0},\rho)\subset\Omega$ there holds $\int%
_{0}^{\tau}\int_{B(x_{0},\rho)}\left\vert \nabla u\right\vert ^{q}%
dxdt=\infty.$ Taking $\psi=\xi^{q^{\prime}}$with $\xi\in\mathcal{D}(\Omega)$,
with $\xi\equiv1$ on $B(x_{0},\rho),$ with values in $\left[  0,1\right]  ,$
we have for any $0<t<\tau,$
\begin{align*}
\int_{B(x_{0},\rho)}u(.,t)dx  &  \geqq\int_{\Omega}u(.,t)\xi^{q^{\prime}%
}dx=\int_{\Omega}u(.,\tau)\xi^{q^{\prime}}dx+\int_{t}^{\tau}\int_{\Omega
}(q^{\prime}\xi^{1/(q-1)}\nabla u.\nabla\xi+\left\vert \nabla u\right\vert
^{q}\xi^{q^{\prime}})dxdt\\
&  \geqq\frac{1}{2}\int_{t}^{\tau}\int_{\Omega}\left\vert \nabla u\right\vert
^{q}\xi^{q^{\prime}})dxdt-C_{q}\int_{t}^{\tau}\int_{\Omega}\left\vert
\nabla\xi\right\vert ^{q^{\prime}}dxdt,
\end{align*}
then
\begin{equation}
\lim_{t\rightarrow0}\int_{B(x_{0},\rho)}u(.,t)dx=\infty. \label{sum}%
\end{equation}
(ii) We define $\mathcal{R}$ as the open set of points $x_{0}\in\Omega$
satisfying $(A1)$ and $\mathcal{S=}\Omega\backslash\mathcal{R}.$ Then from
$(A1),$ there exists a unique measure $u_{0}\in\mathcal{M}(\mathcal{R})$ such
that (\ref{ccc}) holds; and (\ref{ddd}) holds from $(A2).$
\end{proof}

\subsection{First examples}

1) Let $1<q<q_{\ast}.$ (i) The V.S.S. $Y_{\left\{  0\right\}  }$ given by
(\ref{selfs}) in $Q_{\mathbb{R}^{N},\infty}$ admits the trace $(\left\{
0\right\}  ,0).$ \medskip\ 

(ii) Let $\Omega$ be bounded, and $x_{0}\in\Omega$. There exist a weak
solution $Y_{\left\{  x_{0}\right\}  }^{\Omega}$of $(D_{\Omega,\infty})$ with
trace $(\left\{  x_{0}\right\}  ,0)),$ called V.S.S. in $\Omega$ relative to
$x_{0}.$ It is the unique weak solution such that%
\begin{equation}
\lim_{t\rightarrow0}\int_{B(x_{0},\rho)}Y_{\left\{  x_{0}\right\}  }^{\Omega
}(.,t)dx=\infty,\forall\rho>0,\qquad\lim_{t\rightarrow0}\int_{\Omega
}Y_{\left\{  x_{0}\right\}  }^{\Omega}(.,t)\psi dx=0,\forall\psi\in
C_{c}(\overline{\Omega}\backslash\left\{  x_{0}\right\}  ), \label{selo}%
\end{equation}
see \cite[Theorem 1.5]{BiDao1}.\medskip\ 

\noindent2) Let $1<q<q_{\ast}.$ From \cite{QW}, for any $\beta>F(0),$ where
$F$ is defined at (\ref{selfs}), there exists a unique positive radial
self-similar solution $U_{\beta}(x,t)=t^{-a/2}f_{\beta}(\frac{\left\vert
x\right\vert }{\sqrt{t}})$ such that%
\begin{equation}
f_{\beta}(0)=\beta,f_{\beta}^{\prime}(0)=0,\text{ \quad and }\lim
_{\eta\rightarrow\infty}f_{\beta}(\eta)\eta^{a}=C(\beta)>0; \label{cbeta}%
\end{equation}
then $U_{\beta}$ has the trace $(\left\{  0\right\}  ,C(\beta)\left\vert
x\right\vert ^{-a}).$ Notice that $x\mapsto\left\vert x\right\vert ^{-a}$
belongs to $L_{loc}^{1}(\mathbb{R}^{N}\backslash\left\{  0\right\}  )$ but not
to $L_{loc}^{1}(\mathbb{R}^{N}).$ \medskip\ 

\noindent3) Let $q_{\ast}<q<2.$ For any $\beta>0,$ there exists a unique
solution as above, see \cite{QW}. Then $U_{\beta}$ has the trace
$(\emptyset,C(\beta)\left\vert x\right\vert ^{-a});$ notice that
$x\mapsto\left\vert x\right\vert ^{-a}$ belongs to $L_{loc}^{1}(\mathbb{R}%
^{N})$ but not to $L^{1}(\mathbb{R}^{N}).$ \medskip\ 

\noindent4) Let $\Omega$ be bounded, and $q>1.$ From \cite{CLS}, there exists
a solution of $(D_{\Omega,\infty})$ which converges to $\infty$ uniformly on
the compact sets of $\Omega$ as $t\rightarrow0.$ Then its trace is
$(\Omega,0).$ See more details in Section \ref{3}.

\subsection{Lower estimates}

We first give \textit{interior} lower estimates, valid for any $q>1,$ by
constructing a subsolution of the equation, with infinite trace in $B_{1/2}$
and compact support in $B_{1}.$

\begin{proposition}
\label{subsol}Let $q>1,$ and $\Omega$ be any domain in $\mathbb{R}^{N},$ and
let $u$ be any classical solution $u$ of (\ref{un}) in $Q_{\Omega,T},$ such
that $u$ converges uniformly to $\infty$ on a ball $B(x_{0},\rho)\subset
\Omega$, as $t\rightarrow0.$ Then there exists $C=C(N,q,\rho)$ such that
\begin{align}
\lim\inf_{t\rightarrow0}t^{\frac{1}{q-1}}u(x,t)  &  \geqq C=C(N,q,\rho
),\qquad\forall x\in B(x_{0},\frac{\rho}{2}),\label{stx}\\
\lim\inf_{t\rightarrow0}t^{\frac{1}{q-1}}u(x_{0},t)  &  \geqq C_{q}%
\rho^{q^{\prime}},\qquad C_{q}=((q^{\prime}(1+q^{\prime}))^{q}(q-1))^{-\frac
{1}{q-1}}. \label{stv}%
\end{align}

\end{proposition}

\begin{proof}
Let $h,\lambda>0$ be two parameters. We consider a function $t\in\left(
0,\infty\right)  \longmapsto\psi(t)=\psi_{h}(t)\in(1,\infty)$ depending on
$h,$ introduced in \cite{Bi}, solution of the ordinary differential equation%
\begin{equation}
\psi_{t}+h(\psi^{q}-\psi)=0\quad\text{in }\left(  0,\infty\right)  ,\qquad
\psi(0)=\infty,\quad\psi(\infty)=1, \label{jip}%
\end{equation}
given explicitely by $\psi(t)=(1-e^{-h(q-1)t})^{-\frac{1}{q-1}};$ hence
$\psi^{q}-\psi\geqq0,$ and $\psi(t)\geqq(h(q-1)t)^{-1/(q-1)}$ for any $t>0.$
Setting
\[
V(x,t)=\psi(t)f(\left\vert x\right\vert ),\qquad f(r)=(1+q^{\prime
}r)(1-r)^{q^{\prime}},\qquad\forall r\in\left[  0,1\right]  ,
\]
we compute
\[
D=V_{t}-\Delta V+\left\vert \nabla V\right\vert ^{q}-\lambda V=(\left\vert
f^{\prime}\right\vert ^{q}-hf)(\psi^{q}-\psi)+(\left\vert f^{\prime
}\right\vert ^{q}-\Delta f-\lambda f)\psi.
\]
Note that $f^{\prime}(r)=-Mr(1-r)^{q^{\prime}-1},$ with $M=q^{\prime
}(1+q^{\prime}).$ Thus $f^{\prime}(0)=0$ and $f_{0}$ is nonincreasing, and
$\left\vert f^{\prime}\right\vert ^{q}-hf\leqq0$ on $\left[  0,1\right]  $ for
$h\geqq C_{1}=M^{q}.$ Otherwise $\left\vert f^{\prime}\right\vert ^{q}-\Delta
f-\lambda f=(1-r)^{q^{\prime}}J(r)$ with
\[
J(r)=M^{q}r^{q}-\lambda(1+q^{\prime}r)+MG(r),\qquad G(r)=\frac
{N-(N-1+q^{\prime})r}{(1-r)^{2}}.
\]
Then $J(0)=MN-\lambda\beta\leqq0$ for $\lambda\geqq C_{2}=NM.$ We have
\[
J^{\prime}(r)=qM^{q}r^{q-1}-\lambda q^{\prime}+MG^{\prime}(r),\qquad
G^{\prime}(r)=(1-r)^{-3}(N+1-q^{\prime}-(N-1+q^{\prime})r).
\]
If $q\leqq(N+1)/N,$ there holds $q^{\prime}>N+1,$ hence $G^{\prime}\leqq0,$
thus $J^{\prime}\leqq0$, for $\lambda\geqq(q-1)M^{q}.$ If $q>(N+1)/N,$ then
$G^{\prime}(r)\leqq0$ for $r\geqq r_{N,q}=(N+1-q^{\prime})/(N-1+q^{\prime}),$
and $G^{\prime}$ is continuous on $\left[  0,1\right)  $, hence bounded on
$\left[  0,r_{N,q}\right]  .$ Then $J^{\prime}\leqq0$ as soon as $\lambda\geqq
C_{3}=(q-1)M^{q}+(1+q^{\prime})\max_{\left[  0,r_{N,q}\right]  }G^{\prime}.$
We fix $h=h(N,q)\geqq C_{1}$ and $\lambda=\lambda(N,q)\geqq\max(C_{2},C_{3}),$
then $J(r)\leqq0$ on $\left[  0,1\right]  ,$ thus $D\leqq0.$ Then the
function
\[
(x,t)\longmapsto w(x,t)=e^{-\lambda t}V(x,t)=e^{-\lambda t}\psi(t)f_{0}%
(\left\vert x\right\vert )
\]
satisfies
\[
w_{t}-\Delta w+e^{\lambda(q-1)t}\left\vert \nabla w\right\vert ^{q}\leqq0,
\]
hence it is a subsolution of the Dirichlet problem $(D_{B_{1},\infty})$, since
$e^{\lambda(q-1)t}\geqq1.$ By scaling the function $(x,t)\longmapsto\tilde
{w}(x,t)=\rho^{-a}w((x-x_{0})/\rho,t/\rho^{2})$ is a subsolution of
$(D_{B(x_{0},\rho),\infty}).$ And $u$ is a solution in $Q_{\Omega,T}$ which
converges uniformly to $\infty$ on $B(x_{0},\rho)$ as $t\rightarrow0.$ For
given $\epsilon>0,$ there holds $\tilde{w}(.,\epsilon)\leqq m_{\epsilon}%
=\rho^{-a}\psi(\epsilon/\rho^{2})$ in $B(x_{0},\rho);$ and there exists
$\tau_{\epsilon}\in(0,\epsilon)$ such that for any $\theta\in\left(
0,\tau_{\epsilon}\right)  ,$ $u(.,\theta)\geqq m_{\epsilon}$ in $B(x_{0}%
,\rho).$ Then $\tilde{w}(.,t+\epsilon)\leqq u(.,t+\theta)$ in $Q_{B(x_{0}%
,\rho),T-\epsilon}.$ As $\theta\rightarrow0$ and $\epsilon\rightarrow0,$ we
get $\tilde{w}\leqq u$ in $Q_{B(x_{0},\rho),T}.$ And%
\[
\tilde{w}(x,t)\geqq\rho^{-a}e^{-\lambda t/\rho^{2}}\psi(t/\rho^{2})\geqq
(\rho/2)^{q^{\prime}}e^{-\lambda t/\rho^{2}}(h(q-1)t)^{-1/(q-1)}%
\]
in $B(x_{0},\rho/2),$ hence (\ref{stx}) holds. Taking $h=M^{q}=(q^{\prime
}(1+q^{\prime}))^{q},$ there holds $u(x_{0},t)\geqq\rho^{q^{\prime}%
}e^{-\lambda t/\rho^{2}}(h(q-1)t)^{-1/(q-1)},$ thus (\ref{stv})
follows.\medskip
\end{proof}

In case $1<q<q_{\ast},$ we give a lower bound for all the weak solutions at
any singular point, by an argument of stability-concentration, well-known for
semilinear elliptic or parabolic equations, see \cite{MaVe}.

\begin{proposition}
\label{mv} Let $1<q<q_{\ast}$. (i) Let $u$ be any nonnegative weak solution
$u$ (\ref{un}) in $Q_{\mathbb{R}^{N},T}$ with singular set $\mathcal{S}.$ Then
for any $x_{0}\in\mathcal{S}$, there holds $u(x,t)\geqq$ $Y_{\left\{
0\right\}  }(x-x_{0},t)$ in $Q_{\mathbb{R}^{N},T},$ where $Y_{\left\{
0\right\}  }$ is the V.S.S. given at (\ref{selfs}). In particular,
\begin{equation}
u(x_{0},t)\geqq C(N,q)t^{-a/2},\qquad\forall t>0. \label{fer}%
\end{equation}

(ii) Let $\Omega$ bounded, and $u$ be any nonnegative weak solution $u$ of
$(D_{\Omega,T}),$ with singular set $\mathcal{S}.$ Then for any $x_{0}%
\in\mathcal{S}$, $u(x,t)\geqq$ $Y_{\left\{  x_{0}\right\}  }^{\Omega}(x,t)$ in
$Q_{\Omega,T},$ where $Y_{\left\{  x_{0}\right\}  }^{\Omega}$ is given by
(\ref{selo}). In particular,
\begin{equation}
\lim\inf_{t\rightarrow0}t^{\frac{a}{2}}u(x_{0},t)\geqq C(N,q)>0. \label{fes}%
\end{equation}
In any case, $u(.,t)$ converges uniformly on $\mathcal{S}$ to $\infty$ as
$t\rightarrow0.$
\end{proposition}

\begin{proof}
(i) We can assume $x_{0}=0.$ For any $\varepsilon>0,$ there holds
$\lim_{t\rightarrow0}\int_{B_{\varepsilon}}u(x,t)dx=\infty.$ And $u\in
C^{2,1}(Q_{\mathbb{R}^{N},T}).$ We will prove that for fixed $k>0,$ there
holds $u\geqq u^{k},$ where $u^{k}$ is the unique solution in $\mathbb{R}^{N}$
with initial data $k\delta_{0},$ from Theorem \ref{souc}. There exists
$t_{1}>0$ such that $\int_{B_{2^{-1}}}u(x,t_{1})dx>k;$ thus there exists
$s_{1,k}>0$ such that $\int_{B_{2^{-1}}}T_{s_{1,k}}u(x,t_{1})dx=k.$ By
induction, there exists a decreasing sequence $\left(  t_{n}\right)  $
converging to $0,$ and a sequence $\left(  s_{n,k}\right)  $ such that
$\int_{B_{2^{-n}}}T_{s_{n,k}}u(x,t_{n})dx=k.$ Let $p\in\mathbb{N},$ $p>1.$
Denote by $u_{n,k,p}$ the solution of the Dirichlet problem $(D_{B_{p},\infty
})$, with initial data $u_{n,k,p}(.,0)=T_{s_{n,k}}u(.,t_{n})\chi_{B_{2^{-n}}%
}.$ Then we get $u\geqq u_{n,k,p}$ in $B_{p},$ from Corollary \ref{cpri}. As
$n\rightarrow\infty,$ $(u_{n,k,p}(0))$ converges to $k\delta_{0}$ weakly in
$\mathcal{M}_{b}(B_{p}).$ Indeed for any $\psi\in C^{+}(\overline{B_{p}}),$
\[
\left\vert \int_{B_{p}}v_{n,k,p}(0)\psi dx-k\psi(0)\right\vert =\left\vert
\int_{B_{2^{-n}}}T_{s_{n,k}}v(x,t_{n})\psi dx-k\psi(0)\right\vert \leqq
k\left\Vert \psi-\psi(0)\right\Vert _{L^{\infty}(B_{2^{-n}})}.
\]
Then $(u_{n,k,p})$ converges in $C_{loc}^{2,1}(Q_{B_{p},T})$ to the solution
$u^{k,B_{p}}$ of the problem in $B_{p}$ with initial data $k\delta_{0},$ from
Proposition \ref{cpro}. Thus $u\geqq u^{k,B_{p}}.$ Finally, as $p\rightarrow
\infty,$ $u^{k,B_{p}}$ converges to $u^{k}$ from \cite[Lemma 4.6]{BiDao1} and
uniqueness of $u^{k};$ thus $u\geqq u^{k}.$ As $k\rightarrow\infty,$ $(u^{k})$
converges to $Y_{\left\{  0\right\}  },$ hence $v\geqq Y_{\left\{  0\right\}
}.$ Then (\ref{fer}) holds with $C=F(0)$ given by (\ref{selfs}).$\medskip$

(ii) In the same way, denote by $u_{x_{0}}^{k,\Omega},$ $u_{n,k,x_{0}}$ the
solutions of the Dirichlet problem $(D_{\Omega,\infty})$, with respective
initial data $k\delta_{x_{0}}$and $T_{s_{n,k}}v(.,t_{n})\chi_{B_{(x_{0}%
,2^{-n}d)}}$, where $d=d(x_{0},\partial\Omega).$ Then as above we get $u\geqq
u_{n,k,x_{0}}$ in $\Omega,$ then $u\geqq u_{x_{0}}^{k,\Omega}.$ As
$k\rightarrow\infty,$ $(u^{k,\Omega})$ converges to $Y_{\left\{
x_{0}\right\}  }^{\Omega},$ and moreover, for any $\varepsilon>0,$ there
exists $\tau=$ $\tau(\varepsilon,d)$ such that $Y_{\left\{  x_{0}\right\}
}(x,t)=Y(x-x_{0},t)\leqq Y_{\left\{  x_{0}\right\}  }^{\Omega}+\varepsilon$ in
$\overline{\Omega}\times(0,\tau),$ see the proof of \cite[Theorem 1.5]%
{BiDao1}. Then $u\geqq$ $Y_{\left\{  x_{0}\right\}  }^{\Omega}$ and
(\ref{fes}) follows by taking $\varepsilon=F(0)/2.$\medskip
\end{proof}

\begin{remark}
As a consequence, for $1<q<q_{\ast},$ there exists no weak solution $u$ of
(\ref{un}) in $Q_{\mathbb{R}^{N},T}$ with singular set $\mathcal{S}=$
$\mathbb{R}^{N}.$ Indeed if $u$ exists, $u$ satisfies (\ref{fer}), then $u$
converges uniformly on $\mathbb{R}^{N}$ as $t\rightarrow0.$ Then for any $k>0$
and any $\varphi\in\mathcal{D}^{+}(B_{1}),$ $\varphi=1$ in $B_{1/2},$ $u$ is
greater than the solution $u_{k,p}$ with initial trace $k\varphi(x/p).$ As
$p\rightarrow\infty,$ $u_{k,p}$ tends to the unique solution $u_{k}$ with
initial data $k,$ namely $u_{k}\equiv k,$ thus $u\geqq k$ for any $k>0$, which
is contradictory. The question is open for $q\geqq q_{\ast}.$\medskip
\end{remark}

\begin{remark}
Another question is to know for which kind of solutions (\ref{fes}) still
holds when $q\geqq q_{\ast}.$ We give a partial answer in Section \ref{3}, see
Proposition \ref{prec}.
\end{remark}

\subsection{Trace of the Cauchy problem}

In this part we show Theorem \ref{alpha}, based on the universal estimate of
Theorem \ref{fund}.\medskip

\begin{proof}
[Proof of Theorem \ref{alpha}](i) From Theorem \ref{fund}, $u$ satisfies
(\ref{versa}). Reporting in (\ref{un}), we deduce%
\[
u_{t}-\Delta u+\frac{1}{q-1}\frac{u}{t}\geqq0.
\]
Setting $y=t^{1/(q-1)}u,$ we get that
\[
y_{t}-\Delta y=t^{\frac{1}{q-1}}(\frac{1}{q-1}\frac{u}{t}-\left\vert \nabla
u\right\vert ^{q})\geqq0
\]
in $Q_{\mathbb{R}^{N},T},$ thus $y$ has a trace $\gamma\in\mathcal{M}%
^{+}(\mathbb{R}^{N})$, see Lemma \ref{phi}. Since $u(.,t)$ converges weak* to
$u_{0}$ in $\mathbb{R}^{N}\backslash\mathcal{S},$ we find that supp
$\gamma\subset\mathcal{S}.$ Let $B(x_{0},2\eta)\subset\mathbb{R}^{N}%
\backslash\mathcal{S}.$ From (\ref{pluc}), we have
\[
y(x,t)\leqq C(q)\left\vert x-x_{0}\right\vert ^{q^{\prime}}+C(1+t^{q^{\prime}%
}+t^{\frac{1}{q-1}}\int_{B(x_{0},\eta)}du_{0}),\qquad C=C(N,q,\eta),
\]
hence $\gamma\in$ $L_{loc}^{\infty}(\mathbb{R}^{N}).$\medskip
\end{proof}

\begin{remark}
In particular for $q<q_{\ast}$, the V.S.S. $Y_{\left\{  0\right\}  }$ in
$\mathbb{R}^{N}$ satisfies $\gamma=0,$ which can be checked directly, since
lim$_{t\rightarrow0}t^{1/(q-1)-a/2}=0.$ The function $U$ given at Theorem
\ref{szero} satisfies $\gamma(x)=c_{q}(x_{1}^{+})^{q^{\prime}}.$
\end{remark}

\section{Solutions with trace $(\overline{\omega}\cap\Omega,0)$, $\omega$ open
\label{3}}

Here we extend and improve the pionneer result of \cite{CLS}, valid for the
Dirichlet problem in $\Omega$ bounded. In case of the Cauchy problem, the
estimates (\ref{soup}) and (\ref{pluc}) are essential for existence.

\begin{theorem}
\label{xis}Let $q>1$ and $\omega$ be a smooth open set in $\mathbb{R}^{N}$
with $\omega\neq\mathbb{R}^{N}$ (resp. a smooth open set in $\Omega$ bounded).
There exists a classical solution $u=Y_{_{\overline{\omega}}}$ (resp.
$u=Y_{_{\overline{\omega}}}^{\Omega})$ of (\ref{un}) in $Q_{\mathbb{R}%
^{N},\infty}$ (resp. of $(D_{\Omega,\infty})$), with trace $(\overline{\omega
}\cap\Omega,0).$ Moreover it satisfies uniform properties of convergence:%
\begin{equation}
\lim_{t\rightarrow0}\inf_{x\in K}u(x,t)=\infty\quad\forall K\text{ compact
}\subset\omega,\qquad\lim_{t\rightarrow0}\sup_{x\in K}u(x,t)=0\quad\forall
K\text{ compact}\subset\overline{\Omega}\backslash\overline{\omega}.
\label{infsup}%
\end{equation}
More precisely, for any $\delta>0$,
\begin{equation}
\lim\inf_{t\rightarrow0}t^{\frac{1}{q-1}}u(x,t)\geqq C(N,q)\delta^{q^{\prime}%
},\qquad\forall x\in\overline{\omega}_{\delta}^{int}, \label{nej}%
\end{equation}%
\begin{equation}
\sup_{\Omega\backslash\overline{\omega}_{\delta}^{ext}}u(x,t)\leqq
C(N,q,\delta)t,\qquad\forall t>0. \label{naj}%
\end{equation}
\noindent If $q<q_{\ast},$ then for any $x\in\overline{\omega}\cap\Omega,$
\begin{equation}
\inf_{t>0}t^{\frac{a}{2}}u(x,t)\geqq C(N,q)>0\qquad\text{(resp. }\lim
\inf_{t\rightarrow0}t^{\frac{a}{2}}u(x,t)\geqq C(N,q)>0). \label{mini}%
\end{equation}
Moreover, if $\Omega=\mathbb{R}^{N},$ the function $Y_{_{\overline{\omega}}}$
satisfies the growth condition in $Q_{\mathbb{R}^{N},\infty}$
\begin{equation}
Y_{_{\overline{\omega}}}(x,t)\leqq C(t+t^{-\frac{1}{q-1}})(1+\left\vert
x\right\vert ^{q^{\prime}}),\qquad C=C(N,q,\omega) \label{ploi}%
\end{equation}

\end{theorem}

\begin{proof}
First suppose $\Omega$ bounded, then $\overline{\omega}$ is a compact set in
$\mathbb{R}^{N}$. We consider a nondecreasing sequence $(\varphi_{p})$ of
nonnegative functions in $C_{c}^{1}(\Omega)$, with support in $\overline
{\omega},$ such that $\varphi_{p}\geqq p$ in $\overline{\omega}_{1/p}^{int}$,
and the nondecreasing sequence of classical solutions $u_{p}^{\Omega}$ with
initial data $\varphi_{p}$. From Theorem \ref{regdir}, $\left(  u_{p}^{\Omega
}\right)  $ converges in $C_{loc}^{2,1}(Q_{\Omega,T})$ to a solution
$Y_{_{\overline{\omega}}}^{\Omega}$ of $(D_{\Omega,T})$. Then by construction
of $u_{p}^{\Omega},$ $Y_{_{\overline{\omega}}}^{\Omega}(.,t)$ converges
uniformly to $\infty$ on the compact sets in $\omega.$ Then the conclusions
hold from Lemma \ref{bord}, Propositions \ref{subsol} and \ref{mv}.

Next suppose $\Omega=\mathbb{R}^{N}.$ We can construct a nondecreasing
sequence $(\varphi_{p})_{p>p_{0}}$ of functions in $C_{b}^{+}(\mathbb{R}%
^{N}),$ with support in $\overline{\omega}\cap\overline{B_{p}},$ such that
$\varphi_{p}\geqq p$ on $\overline{\omega}_{1/p}^{int}\cap B_{p-1/p}$. Let
$u_{p}$ be the classical solution of (\ref{un}) in $Q_{\mathbb{R}^{N},\infty}$
with initial data $\varphi_{p}$. Since $\omega\neq\mathbb{R}^{N},$ there
exists a ball $B(x_{0},\eta)\subset\mathbb{R}^{N}\backslash\overline{\omega
}).$ From (\ref{pluc}),
\begin{equation}
u_{p}(x,t)\leqq C(q)t^{-\frac{1}{q-1}}\left\vert x-x_{0}\right\vert
^{q^{\prime}}+C(N,q,\eta)(t^{-\frac{1}{q-1}}+t), \label{pln}%
\end{equation}
thus $(u_{p})$ is locally uniformly bounded in $Q_{\mathbb{R}^{N},\infty}.$
From Theorem \ref{raploc}, $(u_{p})$ converges in $C_{loc}^{2,1}%
(Q_{\mathbb{R}^{N},\infty})$ to a classical solution $Y_{_{\overline{\omega}}%
}$ of (\ref{un}) in $Q_{\mathbb{R}^{N},\infty}.$ Then by construction of
$u_{p}$, $Y_{_{\overline{\omega}}}(.,t)$ converges uniformly to $\infty$ on
the compact sets in $\omega,$ and the conclusions follow as above. Moreover,
from (\ref{pln}), $Y_{_{\overline{\omega}}}$ satisfies (\ref{ploi}).\medskip
\end{proof}

\begin{remark}
\label{macs}Moreover, from the construction of the solutions, denoting by
$y_{\varphi}$ the solution of (\ref{un}) with initial data $\varphi\in
C_{b}^{+}(\mathbb{R}^{N})\cap C_{0}^{+}(\overline{\omega})$ (resp. the
solution of $(D_{\Omega,\infty})$ with initial data $\varphi\in C_{0}%
^{+}(\overline{\Omega})$) we get the relations
\begin{equation}
Y_{_{\overline{\omega}}}^{\Omega}=\sup_{\varphi\in C_{0}^{+}(\overline{\Omega
}),\text{supp}\varphi\subset\overline{\omega}}y_{\varphi},\qquad
Y_{_{\overline{\omega}}}=\sup_{\varphi\in C_{b}^{+}(\mathbb{R}^{N}%
),\text{supp}\varphi\subset\overline{\omega}}y_{\varphi}. \label{sphi}%
\end{equation}
Indeed we get $y_{\varphi}\leqq Y_{_{\overline{\omega}}},$ for any nonnegative
$\varphi\in C_{c}^{1}(\mathbb{R}^{N})$ (resp. $C_{c}^{1}(\Omega))$ with
supp$\varphi\subset\overline{\omega},$ and the relation extends to any
$\varphi\in C_{b}^{+}(\mathbb{R}^{N})$ (resp. $C_{0}^{+}(\overline{\Omega})$),
from uniqueness of $y_{\varphi}.$\medskip
\end{remark}

\begin{remark}
When $\Omega$ is bounded, and $\overline{\omega}\subset\Omega,$ or
$\omega=\Omega,$ it was shown in \cite{CLS} that there exists a solution
$Y_{_{\overline{\omega}}}^{\Omega}$ satisfying (\ref{infsup}). Moreover, using
the change of unknown $v=e^{-u},$ they proved that if $\omega\subset
\subset\Omega,$ then for any $x\in$ $\partial\omega,$
\begin{equation}
\lim_{t\rightarrow0}Y_{_{\overline{\omega}}}^{\Omega}(x,t)=\infty,\text{ if
}q<2;\quad\lim_{t\rightarrow0}Y_{_{\overline{\omega}}}^{\Omega}(x,t)=\ln
2,\text{ if }q=2;\quad\lim_{t\rightarrow0}Y_{_{\overline{\omega}}}^{\Omega
}(x,t)=0,\text{ }\qquad\text{if }q>2. \label{pac}%
\end{equation}

\end{remark}

Next we study the question of the \textit{uniqueness} of solutions with trace
$(\overline{\omega},0)$ which appears to be delicate. A first point is to
precise in what class of solutions the uniqueness may hold, in particular in
what sense the initial data are achieved.

\begin{definition}
Let $\Omega=\mathbb{R}^{N}$ (resp. $\Omega$ bounded) and $\omega$ be a open
set in $\Omega.$ We denote by $\mathcal{C}$ the class of classical solutions
of (\ref{un}) in $Q_{\mathbb{R}^{N},T}$ (resp. of $(D_{\Omega,T})$) satisfying
(\ref{infsup}). We denote by $\mathcal{W}$ the class of weak solutions of
(\ref{un}) in $Q_{\mathbb{R}^{N},T}$ (resp. of $(D_{\Omega,T})$) with trace
$(\overline{\omega},0).$\medskip
\end{definition}

In \cite{CLS}, the authors consider the class $\mathcal{C}$. They show that if
$\overline{\omega}$ is compact contained in $\Omega$ bounded and
$\omega,\Omega$ are starshaped with respect to the same point or $q\geqq2$,
then $Y_{_{\overline{\omega}}}^{\Omega}$ is unique in that class. But we
cannot ensure that \textit{any weak solution} $u$ with trace $(\overline
{\omega},0)$ converges \textit{uniformly} to $\infty$ on the compact sets in
$\omega.$ And in case $q>2$ we even do not know if $u$ is continuous. Here we
give some partial results, where we do not suppose that $\Omega$ is starshaped.

\begin{theorem}
\label{maxmin} (i) Let $q>1.$ Under the assumptions of Theorem \ref{xis},
$Y_{_{\overline{\omega}}}=\sup Y_{_{\overline{\omega}_{\delta}^{int}}}$ and
$Y_{_{\overline{\omega}}}$ is a minimal solution in the class $\mathcal{C}$
(resp. $Y_{_{\overline{\omega}}}^{\Omega}=\sup Y_{_{\overline{\omega}_{\delta
}^{int}}}^{\Omega}$ and $Y_{_{\overline{\omega}}}^{\Omega}$ is a minimal
solution in the class $\mathcal{C}$). If $\omega$ is compact, $\overline
{u}_{_{\overline{\omega}}}=\inf_{\delta>0}Y_{_{\overline{\omega}_{\delta
}^{ext}}}$ is a maximal solution of (\ref{un}) in $Q_{\mathbb{R}^{N},T}$ in
the class $\mathcal{C}$ (resp. if $\omega\subset\subset\Omega,$ then
$\overline{u}_{_{\overline{\omega}}}^{\Omega}=\inf_{\delta}Y_{_{\overline
{\omega}_{\delta}^{ext}}}^{\Omega}$ is a maximal solution of $(D_{\Omega,T})$
in the class $\mathcal{C}$).\medskip

(ii) Let $1<q\leqq2$ and suppose $\omega$ compact (resp. $\omega\subset
\subset\Omega).$ Then the function $\overline{u}_{_{\overline{\omega}}}($resp.
$\overline{u}_{_{\overline{\omega}}}^{\Omega})$ defined above is maximal in
the class $\mathcal{W}.$ If $\omega$ is starshaped, then $Y_{_{\overline
{\omega}}}$ (resp. $Y_{_{\overline{\omega}}}^{\Omega}$) is the unique solution
of (\ref{un}) in $Q_{\mathbb{R}^{N},T}$ (resp. of $(D_{\Omega,T}))$ in the
class $\mathcal{C}.$\medskip

(iii) Let $1<q<q_{\ast}$ and suppose $\omega$ compact (resp. $\omega
\subset\subset\Omega).$ Then $\mathcal{W=C}.$ Thus $Y_{_{\overline{\omega}}}$
(resp. $Y_{_{\overline{\omega}}}^{\Omega}$) is minimal in the class
$\mathcal{W}.$ If $\omega$ is starshaped it is unique in the class
$\mathcal{W}$.
\end{theorem}

\begin{proof}
(i) Let $\Omega$ be bounded. Let $v$ be any classical solution of
$(D_{\Omega,T})$ satisfying (\ref{infsup}). Let $\varphi\in C_{0}%
^{+}(\overline{\Omega})$ with supp$\varphi\subset\overline{\omega}.$ Then
there exists a nondecreasing sequence $(\varphi_{n})\in C_{0}^{+}%
(\overline{\Omega}),$ with support in $\overline{\omega}_{1/n}^{int},$
converging to $\varphi$ in $C^{+}(\overline{\Omega}).$ Then the function
$(y_{\varphi_{n}}(.,t)))$ defined at Remark \ref{macs} converges to
$y_{\varphi}(.,t)$ in $C(\overline{\Omega}),$ uniformly for $t>0.$ For fixed
$n,$ let $\epsilon\in\left(  0,T\right)  .$ Since $v(.,t)$ converges uniformly
to $\infty$ on the compact sets of $\overline{\omega},$ and $\varphi_{n}=0$ in
$\overline{\Omega}\backslash\overline{\omega}_{1/n}^{int},$ there exists
$\theta_{n}\in\left(  0,\epsilon\right)  $ such that $\inf v(.,t)\geqq
\max\varphi_{n}\geqq\max y_{\varphi_{n}}(.,\epsilon)$ for any $t\leqq
\theta_{n}.$ Then $v\geqq y_{\varphi_{n}}$ on $\left[  \epsilon,T\right)  $
from the comparison principle, hence $v\geqq y_{\varphi}.$ Then
$Y_{_{\overline{\omega}}}^{\Omega}$ is minimal in the class $\mathcal{C}.$
Moreover for any $\delta>0,$ $Y_{_{\overline{\omega}_{\delta}^{int}}}^{\Omega
}\leqq Y_{_{\overline{\omega}}}^{\Omega},$ and from (\ref{sphi}),
\[
\sup Y_{_{\overline{\omega}_{\delta}^{int}}}^{\Omega}=\sup_{\delta}%
(\sup_{\varphi\in C_{0}^{+}(\overline{\Omega}),\text{supp}\varphi
\subset\overline{\omega}_{\delta}^{int}}y_{\varphi})=\sup_{\varphi\in
C_{0}^{+}(\overline{\Omega}),\text{supp}\varphi\subset\overline{\omega}%
}y_{\varphi}=Y_{_{\overline{\omega}}}^{\Omega}.
\]
Now consider the case $\Omega=\mathbb{R}^{N}.$ Let $v$ be any classical
solution in $Q_{\mathbb{R}^{N},T}$ satisfying (\ref{infsup}). Let $\varphi\in
C_{c}^{+}(\mathbb{R}^{N}),$ with supp$\varphi\subset\overline{\omega}.$ As
above we deduce that $v\geqq y_{\varphi}.$ From the uniqueness of the
solutions, we deduce that $v\geqq y_{\varphi},$ for any $\varphi\in C_{b}%
^{+}(\mathbb{R}^{N}),$ with supp$\varphi\subset\overline{\omega}.$ Then
$Y_{_{\overline{\omega}}}$ is minimal in the class $\mathcal{C}.$ As above we
obtain $Y_{_{\overline{\omega}}}=\sup Y_{_{\overline{\omega}_{\delta}^{int}}%
}.$

Assume that $\Omega=\mathbb{R}^{N}$ and $\overline{\omega}$ is compact. For
$\delta>0$ we consider the function $Y_{_{\overline{\omega}_{\delta}^{ext}}}$
constructed as above$.$ Then by construction, $Y_{_{\overline{\omega}}}\leqq
Y_{_{\overline{\omega}_{\delta}^{ext}}}$. Taking $\delta_{n}\rightarrow0,$
$(Y_{_{\overline{\omega}_{\delta_{n}}^{ext}}})$ decreases to a classical
solution $\overline{u}_{_{\overline{\omega}}}$ of (\ref{un}) in $Q_{\mathbb{R}%
^{N},T}$ from Theorem \ref{raploc} thus $\overline{u}_{_{\overline{\omega}}%
}\geqq Y_{_{\overline{\omega}}},$ then $\overline{u}_{_{\overline{\omega}}}$
satisfies (\ref{infsup}). Moreover let $v$ be any solution in the class
$\mathcal{C}$. From Lemma \ref{bord} (ii), $v\leqq Y_{_{\overline{\omega
}_{\delta}^{ext}}},$ then $v\leqq\overline{u}_{_{\overline{\omega}}},$ thus
$\overline{u}_{_{\overline{\omega}}}$ is maximal. Next assume $\Omega$ bounded
and $\omega\subset\subset\Omega$; the result follows as above by taking
$\delta<\delta_{0}$ small enough such that $\overline{\omega}_{\delta_{0}%
}^{ext}\subset\Omega$ and using Theorem \ref{regdir}.\medskip

(ii) For $q\leqq2,$ $\overline{u}_{_{\overline{\omega}}}($resp. $\overline
{u}_{_{\overline{\omega}}}^{\Omega})$ is also maximal in the class
$\mathcal{W},$ from Lemma \ref{bord} (iii). But we cannot ensure that is
minimal in this class.

Suppose that $\omega$ is starshaped, then $Y_{_{\overline{\omega}}}%
(x,t)=k^{a}Y_{k\overline{\omega}}(kx,k^{2}t),$ from (\ref{sphi}). As above,
any weak solution $v$ of (\ref{un}) in $Q_{\mathbb{R}^{N},T}$ with trace
$(\overline{\omega},0)$ satisfies $v\leqq Y_{k\overline{\omega}}$ for any
$k>1,$ hence $v\leqq Y_{\overline{\omega}}$ as $k\rightarrow1,$ thus
$\overline{u}_{_{\overline{\omega}}}\leqq Y_{\overline{\omega}},$ hence
$\overline{u}_{_{\overline{\omega}}}=Y_{\overline{\omega}}.$ We get uniqueness
in the class $\mathcal{C}$. Now any weak solution $w$ of $(D_{\Omega,T})$ with
trace $(\overline{\omega},0)$ also satisfies $w\leqq Y_{k\overline{\omega}}$
in $\overline{\Omega}\times(0,T)$ for any $k>1,$ then also $Y_{_{\overline
{\omega}}}^{\Omega}\leqq\overline{u}_{\omega}^{\Omega}\leqq Y_{k\overline
{\omega}}.$ Thus as $k\rightarrow1,$ one gets $Y_{_{\overline{\omega}}%
}^{\Omega}\leqq\overline{u}_{\omega}^{\Omega}\leqq Y_{_{\overline{\omega}}}.$
Let $\epsilon_{0}>0.$ We fix $\delta>0$ such that $\overline{\omega}_{\delta
}^{ext}\subset\Omega$. From Lemma \ref{bord} (i), we get $Y_{_{\overline
{\omega}}}(.,t)\leqq C(N,q,\delta)t$ on $\partial\Omega;$ hence there exists
$\tau_{0}>0$ such that $Y_{\overline{\omega}}\leqq\epsilon_{0}$ on
$\partial\Omega\times\left(  0,\tau_{0}\right]  $; thus, for any $\eta<1,$
$Y_{\eta\overline{\omega}}\leqq Y_{_{\overline{\omega}}}^{\Omega}+\epsilon
_{0},$ in $\overline{\Omega}\times\left(  0,\tau_{0}\right]  .$ As
$\eta\rightarrow1$ we get $Y_{_{\overline{\omega}}}\leqq Y_{_{\overline
{\omega}}}^{\Omega}+\epsilon_{0},$ in $\overline{\Omega}\times\left(
0,\tau_{0}\right]  .$ Then $\overline{u}_{\omega}^{\Omega}\leqq Y_{_{\overline
{\omega}}}^{\Omega}+\epsilon_{0},$ in $\overline{\Omega}\times\left(
0,\tau_{0}\right]  .$ From the comparison principle, $\overline{u}_{\omega
}^{\Omega}\leqq Y_{_{\overline{\omega}}}^{\Omega}+\epsilon_{0},$ in
$\overline{\Omega}\times(0,T).$ As $\epsilon_{0}\rightarrow0$ we get
$\overline{u}_{\omega}^{\Omega}\leqq Y_{_{\overline{\omega}}}^{\Omega},$ hence
$\overline{u}_{\omega}^{\Omega}=Y_{_{\overline{\omega}}}^{\Omega}.$ And any
weak solution $v$ of (\ref{un}) with trace $(\overline{\omega},0)$ satisfies
$v\leqq Y_{k\overline{\omega}}$ in $Q_{\mathbb{R}^{N},T},$ for any $k>1;$ thus
as $k\rightarrow1,$ $\overline{u}_{\omega}\leqq Y_{_{\overline{\omega}}},$
hence $\overline{u}_{\omega}=Y_{_{\overline{\omega}}}.$\medskip

(iii) Any weak solution $v\in\mathcal{W}$ is classical since $q\leqq2,$ and
from Proposition \ref{mv}, $v(.,t)$ converges uniformly in $\overline{\omega}$
to $\infty$ as $t\rightarrow0.$ Then $\mathcal{W=C}.$ the conclusions follow
from (i) and (ii).\medskip
\end{proof}

As a consequence we construct the solution of Theorem \ref{szero}. We are lead
to the case $N=1.$

\begin{proposition}
\label{selfdys}Let $q>1,$ $N=1.$ Then there exists a self-similar positive
solution $U(x,t)=t^{-a/2}f(t^{-1/2}x)$ of (\ref{un}) in $Q_{\mathbb{R},T}$,
with trace $(\left[  0,\infty\right)  ,0),$ and $f$ satisfies the equation
\begin{equation}
f^{\prime\prime}(\eta)+\frac{\eta}{2}f^{\prime}(\eta)+\frac{a}{2}%
f(\eta)-\left\vert f^{\prime}(\eta)\right\vert ^{q}=0,\qquad\forall\eta
\in\mathbb{R}. \label{eqf}%
\end{equation}
And setting $c_{q}=(q^{\prime})^{-q^{\prime}}(q-1)^{-1/(q-1)}$,
\begin{equation}
\lim_{\eta\rightarrow\infty}\eta^{-q^{\prime}}f(\eta)=c_{q}, \label{pui}%
\end{equation}%
\begin{equation}
\lim_{\eta\rightarrow-\infty}e^{\frac{\eta^{2}}{4}}(-\eta)^{-\frac{3-2q}{q-1}%
}f(\eta)=C>0. \label{exp}%
\end{equation}
In case $q=2,$ $f$ is given explicitely by
\begin{equation}
f(\eta)=-\ln(\frac{1}{2}erfc(\eta/2))=-\ln(\frac{1}{2}\int_{\eta/2}^{\infty
}e^{-s^{2}}ds). \label{erf}%
\end{equation}

\end{proposition}

\begin{proof}
We apply Theorems \ref{xis} and \ref{maxmin} with $\Omega=\mathbb{R}$ and
$\omega=(0,\infty).$ Since $\omega$ is starshaped and stable by homothety, we
have $Y_{_{\overline{\omega}}}(x,t)=k^{a}Y_{k\overline{\omega}}(kx,k^{2}%
t)=k^{a}Y_{\overline{\omega}}(kx,k^{2}t)$ for any $k>0.$ Thus
$U=Y_{_{\overline{\omega}}}$ is self-similar. Hence $U(x,t)=t^{-a/2}%
f(t^{-1/2}x),$ where $\eta\longmapsto f(\eta)$ is a nonnegative $C^{2}%
$-function on $\mathbb{R}$ and satisfies equation (\ref{eqf}).\medskip

In the case $q=2,$ we can compute completely $U:$ The function $V=e^{-U}$ is
solution of the heat equation, with $V(0,x)=\chi_{\left(  -\infty,0\right)
},$ thus
\[
V(t,x)=(4\pi t)^{-1/2}\int_{-\infty}^{0}e^{-\frac{(x-y)^{2}}{4t}}dy=\frac
{1}{2}\text{erfc}(\frac{x}{2\sqrt{t}})
\]
where $x\longmapsto$ erfc$(x)=\frac{2}{\sqrt{\pi}}\int_{x}^{\infty}e^{-s^{2}%
}ds$ is the complementary error function. Then $U(x,t)=-\ln V,$ and $f$ is
given by (\ref{erf}). Note that $f$ can also be obtained by solving equation
$f^{\prime\prime}(\eta)+\frac{\eta}{2}f^{\prime}(\eta)-f^{\prime}(\eta
)^{2}=0,$ of the first order in $f^{\prime}.$ We get $f(0)=\ln2.$ As
$\eta\rightarrow\infty,$ since erfc$(x)=(1/\sqrt{\pi}x)e^{-x^{2}}(1+o(1),$ we
check that $f(\eta)=(1/4)\eta^{2}(1+o(1)).\medskip$

Next suppose $q\neq2.$ Writing (\ref{eqf}) as a system%
\[
f^{\prime}(\eta=g(\eta),\qquad g^{\prime}(\eta)=-\frac{\eta}{2}g(\eta
)-\frac{a}{2}f(\eta)+\left\vert g(\eta)\right\vert ^{q},
\]
we obtain that $f$ is positive, from the Cauchy-Lipschitz Theorem. Indeed if
there holds $f(\eta_{1})=0$ for some $\eta_{1}\in\mathbb{R},$ then $g(\eta
_{1})=f^{\prime}(\eta_{1})=0,$ thus $(f,g)\equiv(0,0).$ From (\ref{stx}), we
get $U(1,t)=t^{-a/2}f(t^{-1/2})\geqq Ct^{1/(q-1),}$ for $t$ small enough,
hence $f(\eta)\geqq C\eta^{q^{\prime}}$ for large $\eta.$ From (\ref{expo}),
there holds $U(-1,t)\leqq C_{1,1}e^{-C_{2,1}/t}$ on $\left(  0,\tau
_{1}\right]  ,$ since $U$ is a pointwise limit of classical solutions with
initial data $C_{b}(\mathbb{R})$ with support in $\left[  0,\infty\right)  .$
Then $f(\eta)$ converges to $0$ exponentially as $\eta\rightarrow-\infty.$
Next we show that $f^{\prime}>0$ on $\mathbb{R}:$ if $f^{\prime}(\eta_{0})=0$
for some $\eta_{0}$ we have $f^{\prime\prime}(\eta_{0})+\frac{a}{2}f(\eta
_{0})=0.$ Since $a\neq0,$ $\eta_{0}$ is unique, it is a strict local extremum,
which contradicts the behaviour at $\infty$ and $-\infty$. The universal
estimate (\ref{versa}) is equivalent to
\begin{equation}
f^{\prime q}(\eta)\leqq\frac{1}{q-1}f(\eta),\qquad\forall\eta\in\mathbb{R}.
\label{visa}%
\end{equation}
Therefore the function $\eta\longmapsto f^{1/q^{\prime}}(\eta)-c_{q}%
^{1/q^{\prime}}\eta$ is nonincreasing, hence
\begin{equation}
f^{1/q^{\prime}}(\eta)\leqq c_{q}^{1/q^{\prime}}\eta+f^{1/q^{\prime}%
}(0),\qquad\forall\eta\geqq0. \label{pai}%
\end{equation}
Otherwise, $f$ is convex: indeed
\begin{equation}
f^{\prime\prime\prime}+\frac{\eta}{2}f^{\prime\prime}+\frac{1}{2(q-1)}%
f^{\prime}-qf^{\prime q-1}f^{\prime\prime}=0. \label{plo}%
\end{equation}
If $f^{\prime\prime}(\eta_{1})=0$ for some $\eta_{1},$ then $f^{\prime
\prime\prime}(\eta_{1})<0,$ thus $\eta_{1}$ is unique, and $f^{\prime\prime
}(\eta)<0$ for $\eta>\eta_{1},$ then $f$ is concave near $\infty$, which
contradicts the estimates above; thus $f^{\prime\prime}(\eta)>0$ on
$\mathbb{R}.$ From (\ref{eqf}) and (\ref{visa}), we deduce that $\eta
f^{\prime}\leqq q^{\prime}f.$

Let $H(\eta)=\eta^{-q^{\prime}}f(\eta),$ for $\eta>0;$ then $H$ is
nonincreasing, and $H(\eta)\geqq C$ for large $\eta.$ Thus $H$ has a limit
$\lambda>0$ as $\eta\rightarrow\infty,$ and $\lambda\leqq c_{q}$ from
(\ref{pai}). Let us show that $\lambda=c_{q}.$ Suppose that $\lambda<c_{q}.$
We set $\varphi(\eta)=\eta^{-1/(q-1)}f^{\prime}(\eta),$ for $\eta>0,$ then
$\varphi\leqq q^{\prime}H;$ hence we can find $b<1$ such that $q\varphi
^{q-1}(\eta)<b$ for large $\eta.$ By computation we find
\begin{equation}
\frac{1}{\eta}\varphi^{\prime}=\varphi^{q}-\varphi(\frac{1}{2}+\frac
{1}{(q-1)\eta^{2}})-\frac{a}{2}H, \label{tr}%
\end{equation}
and from (\ref{plo}) we obtain
\[
\varphi^{\prime\prime}+\varphi^{\prime}(\frac{2}{(q-1)\eta}+\frac{\eta}%
{2}-q\eta\varphi^{q-1})+\frac{\varphi}{q-1}(1-q\varphi^{q-1}+\frac{a}{\eta
^{2}})=0
\]
If $\varphi$ is not monotone for large $\eta$, then, at any extremal point
$\eta,$
\[
-\varphi^{\prime\prime}=\frac{\varphi}{q-1}(1-q\varphi^{q-1}+\frac{a}{\eta
^{2}})\geqq\frac{\varphi}{q-1}(1-b+\frac{a}{\eta^{2}}),
\]
hence $\varphi^{\prime\prime}<0$ for large $\eta,$ which is impossible. Thus
by monotony, $\varphi$ has a limit $\theta$ as $\eta\rightarrow\infty.$ From
the L'Hospital's rule, we deduce that $\lambda=\lim_{\eta\rightarrow\infty
}f(\eta)/\eta^{q^{\prime}}=\lim_{\eta\rightarrow\infty}f^{\prime}%
(\eta)/q^{\prime}\eta^{1/(q-1)}=\theta/q^{\prime}$. Then from (\ref{tr}),
$\lim_{\eta\rightarrow\infty}\varphi^{\prime}(\eta)/\eta=(q^{\prime}%
\lambda)^{q}-\lambda/(q-1).$ Since $\varphi^{\prime}$ is integrable, we deduce
that $\lambda=c_{q},$ thus we reach a contradiction. Then (\ref{pui})
follows.\medskip

Next we study the behaviour near $-\infty.$ From (\ref{visa}), $f$ and
$f^{\prime}$ converge exponentially to $0$. Let $h(\eta)=f^{\prime}%
(\eta)/f(\eta)$ for any $\eta\in\mathbb{R}.$ Then we find
\begin{equation}
h^{\prime}+h^{2}+\frac{\eta}{2}h+\frac{a}{2}-f^{\prime(q-1)}h=0, \label{aqh}%
\end{equation}%
\[
h^{\prime\prime}+2hh^{\prime}+\frac{\eta}{2}h^{\prime}+\frac{h}{2}%
-f^{\prime(q-1)}(qh^{\prime}+(q-1)h^{2})=0.
\]
Either $h$ is not monotone near $-\infty$. At any point where $h^{\prime}=0,$
we find by computation
\[
h^{\prime\prime}=(q-1)h(h(h+\frac{\eta}{2})-\frac{1}{2});
\]
hence at any minimal point, $h>\left\vert \eta\right\vert /2,$ then
$\lim_{\eta\rightarrow-\infty}h(\eta)=\infty.$ Let us show that it also true
if $h$ is monotone. Suppose that $h$ has a finite limit $\ell,$ then $\ell=0$
from (\ref{aqh}). If $q>2,$ then $\lim\inf_{\eta\rightarrow-\infty}h^{\prime
}(\eta)\geqq\left\vert a\right\vert /2,$ which is contradictory. If $q<2,$
following the method of \cite{BrPeTe} we write $(e^{\eta^{2}/4}h)^{\prime
}=e^{\eta^{2}/4}(-a/2+o(1)),$ then by integration we obtain that $\lim
_{\eta\rightarrow-\infty}\eta h(\eta)=a,$ from the l'Hospital' rule, then
$\lim\inf_{\eta\rightarrow\infty}\left(  -\eta\right)  ^{a}f(\eta)>0,$ which
is a contradiction. Thus again $\lim_{\eta\rightarrow-\infty}h(\eta)=\infty.$
And then (\ref{exp}) follows as in \cite{BrPeTe}, more precisely, as
$\eta\rightarrow-\infty,$
\[
f(\eta)=Ce^{\frac{-\eta^{2}}{4}}\left\vert \eta\right\vert ^{\frac{3-2q}{q-1}%
}(1-(a-1)(a-2)\left\vert \eta\right\vert ^{-2}+o(\left\vert \eta\right\vert
^{-2}).
\]

\end{proof}

Thanks to the barrier function $U(x,t)=t^{-a/2}f(t^{-1/2}x)$ constructed at
Proposition \ref{selfdys}, we obtain more information on the behaviour of the
solutions with trace $(\overline{\omega},0)$ on the boundary of $\omega:$

\begin{proposition}
\label{prec}Let $1<q$, and $\omega$ be a smooth open set in $\mathbb{R}^{N}.$
Then the function $Y_{\overline{\omega}}$ constructed at Theorem \ref{xis} satisfies

(i) For any $x_{0}\in\partial\omega,$ $\qquad\lim\inf_{t\rightarrow0}%
t^{a/2}Y_{\overline{\omega}}\left(  x_{0},t\right)  \geqq f(0).\medskip$

(ii) If $\omega$ is convex, then for any $x_{0}\in\partial\omega,\qquad$
$\lim_{t\rightarrow0}t^{a/2}Y_{\overline{\omega}}\left(  x_{0},t\right)
=f(0),\medskip$

(iii) if $\mathbb{R}^{N}\backslash\omega$ is convex, then for any $x_{0}%
\in\overline{\omega},\qquad\inf_{t>0}t^{a/2}Y_{\overline{\omega}}\left(
x_{0},t\right)  \geqq f(0),$

\noindent where $f$ is defined at Proposition \ref{selfdys}.
\end{proposition}

\begin{proof}
(i) Since $\omega$ is smooth, it satisfies the condition of the interior
sphere. Thus we can assume that $x_{0}=0$ and $\omega$ contains a ball
$B=B(y,\rho)$ with $y=(\rho,0)\in\mathbb{R}^{N+}=\mathbb{R}^{+}\mathbb{\times
R}^{N-1}.$ Then $Y_{\overline{\omega}}\geqq Y_{\overline{B}}.$ Let us consider
$Y_{\overline{nB}}(x,t)=n^{-a}Y_{\overline{B}}(x/n,t/n^{2}).$ The sequence
$\left(  Y_{\overline{nB}}\right)  $ is nondecreasing, and there holds
$Y_{\overline{nB}}(x,t)=0$ in $B((-1,0),1).$ Thus from estimate (\ref{pluc}),
\[
Y_{\overline{nB}}(x,t)\leqq C(N,q)(t^{-\frac{1}{q-1}}(\left\vert
x+(1,0)\right\vert ^{q^{\prime}}+1)+t),
\]
hence the sequence is locally bounded in $Q_{\mathbb{R}^{N},\infty}.$ From
Theorem \ref{raploc}, $\left(  Y_{\overline{nB}}\right)  $ converges in
$C_{loc}^{2,1}(Q_{\mathbb{R}^{N},\infty})$ to a classical solution $u$ of
(\ref{un}). Then $u$ is a solution with trace ($\overline{\mathbb{R}^{N+}%
},0),$ satisfying (\ref{infsup}), thus $u(x,t)\geqq Y_{\overline
{\mathbb{R}^{N+}}}(x,t).$ Observe that $Y_{\overline{\mathbb{R}^{N+}}%
}(x,t)=U(x_{1},t),$ since $U(x_{1},t)=\sup_{\varphi\in C_{b}^{+}%
(\mathbb{R}),\text{supp}\varphi\subset\overline{0,\infty}}y_{\varphi},$ and
$Y_{\overline{\mathbb{R}^{N+}}}(x,t)=\sup_{\varphi\in C_{c}^{+}(\mathbb{R}%
^{N}),\text{supp}\varphi\subset\overline{\mathbb{R}^{N+}}}y_{\varphi}.$ Then
$u(0,t)\geqq U(0,t)=f(0)t^{-a/2}.$ And $Y_{\overline{nB}}(0,1)=n^{-a}%
Y_{\overline{B}}(0,1/n^{2})$ converges to $u(0,1)\geqq f(0),$ then
lim$n^{-a}Y_{\overline{B}}(0,1/n^{2})\geqq$ $f(0)$; similarly by replacing
$1/n$ by any sequence $(\epsilon_{n})$ decreasing to $0,$ then $\lim
\inf_{t\rightarrow0}t^{a/2}Y_{\overline{\omega}}\left(  0,t\right)  \geqq
f(0).$\medskip

(ii) Let us show that for any $x_{0}\in\partial\omega,$ $Y_{\overline{\omega}%
}\left(  x_{0},t\right)  \leqq f(0)t^{-a/2}.$ We can assume $x_{0}=0$ and
$\omega\subset\mathbb{R}^{N+}.$ Then $Y_{\overline{\omega}}(x,t)\leqq
Y_{\overline{\mathbb{R}^{N+}}}(x,t)=U(x_{1},t)$, hence $Y_{\overline{\omega}%
}\left(  0,t\right)  \leqq f(0)t^{-a/2}.$\medskip

(iii) Since $\mathbb{R}^{N}\backslash\omega$ is convex, $\overline{\omega}$ is
the union of all the tangent half-hyperplanes that it contains. For any such
half-hyperplane, we can assume that it is tangent at $0$ and equal to
$\mathbb{R}^{N+}.$ Then for any $x\in\mathbb{R}^{N+}$, there holds
$Y_{\overline{\omega}}\left(  x,t\right)  \geqq U(x_{1},t)\geqq f(0),$ since
$f$ is nondecreasing, and the conclusion follows.
\end{proof}

\section{Existence of solutions with trace $(S,u_{0})$\label{4}}

\subsection{Solutions with trace $(\overline{\omega}\cap\Omega,u_{0})$,
$\omega$ open}

\begin{proof}
[Proof of Theorem \ref{openex}](i) Approximation and convergence. We define
suitable approximations of the initial trace $(\mathcal{S},u_{0})$ according
to the value of $q.$ We consider a sequence $(\varphi_{p})$ in $C_{b}\left(
\mathbb{R}^{N}\right)  $ (resp. $C_{0}\left(  \overline{\Omega}\right)  )$ as
in the proof of Theorem \ref{xis}. We define a sequence $(\psi_{p})$ in the
following way: if $1<q<q_{\ast},$ we define $\psi_{p}$ by the restriction of
the measure $u_{0}$ to $\mathcal{R}_{1/p}^{int}\cap B_{p}$ (resp. to
$\mathcal{R}_{1/p}^{int}\cap\Omega_{1/p}^{int});$ if $q_{\ast}\leqq q\leqq2,$
we take $\psi_{p}=\inf(u_{0},p)\chi_{\mathcal{R}\cap B_{p}}$ (resp. $\psi
_{p}=\inf(u_{0},p)\chi_{\mathcal{R}}$). If $q>2,$ by our assumption we can
take a nondecreasing sequence $(\psi_{p})$ in $C_{c}\left(  \mathcal{R}%
\right)  $ converging to $u_{0}$ in $L_{loc}^{1}\left(  \mathcal{R}\right)  .$
We set $u_{0,p}=\varphi_{p}+\psi_{p}.$ Then for $1<q<q_{\ast},$ $u_{0,p}%
\in\mathcal{M}_{b}^{+}(\Omega)$, for $q_{\ast}\leqq q\leqq2,$ $u_{0,p}\in
L^{r}(\Omega)$ for any $r>1$ and for $q>2,$ $u_{0,p}\in C_{b}\left(
\mathbb{R}^{N}\right)  .$ In any case there exists a solution $u_{p}$ of
(\ref{un}) (resp. of $(D_{\Omega,T})$) with initial data $u_{0,p},$ unique
among the weak solutions if $q\leqq2,$ see Theorem \ref{souc}, and among the
classical solutions in $C\left(  \left[  0,T\right)  \times\overline{\Omega
}\right)  $ if $q>2$, and the sequence $(u_{p})$ is nondecreasing if $q\geqq
q_{\ast}.$

Moreover if $\Omega=\mathbb{R}^{N},$ $(u_{p})$ satisfies the estimate
(\ref{pluc}): considering a ball $B(x_{0},\eta)\subset\mathbb{R}^{N}%
\backslash\overline{\omega}$, there exists $C=C(N,q,\eta)$ such that for
$p\geqq p(\eta)$ large enough,
\[
u_{p}(x,t)\leqq C(t^{-\frac{1}{q-1}}(\left\vert x-x_{0}\right\vert
^{q^{\prime}}+1)+t+\int_{B(x_{0},\eta)}du_{0,p})\leqq C(t^{-\frac{1}{q-1}%
}(\left\vert x-x_{0}\right\vert ^{q^{\prime}}+1)+t+\int_{B(x_{0},\eta)}%
du_{0}),
\]
then $(u_{p})$ is uniformly locally bounded in $Q_{\mathbb{R}^{N},T}$ ( resp.
if $\Omega$ is bounded, $(u_{p})$ satisfies (\ref{wo}), since it is
constructed by approximation from solutions with smooth initial data). From
Theorem \ref{raploc} (resp. \ref{regdir})), we can extract a subsequence
$C_{loc}^{2,1}$-converging to a classical solution $u$ of (\ref{un}) in
$Q_{\mathbb{R}^{N},T}$ (resp. of $(D_{\Omega,T})$). If $q\geqq q_{\ast},$ from
uniqueness, $(u_{p})$ is nondecreasing, then $(u_{p})$ converges to $u=\sup
u_{p}.\medskip$

(ii) Behaviour of $u$ in $\overline{\omega}.$ By construction, $u\geqq
Y_{\overline{\omega}},$ (resp. $u\geqq Y_{\overline{\omega}}^{\Omega}),$ then
$u$ satisfies (\ref{nej}), hence as $t\rightarrow0,$ $u(.,t)$ converges
uniformly to $\infty$ on any compact in $\omega,$ thus (\ref{ddd}) holds; if
$q<q_{\ast},$ $u$ satisfies (\ref{mini}), thus the convergence is uniformly on
$\overline{\omega}\cap\Omega.\medskip$

(iii) Behaviour of $u$ in $\mathcal{R}$\textbf{$.$ }From (\ref{zif}) and
(\ref{bou}), for any $\xi\in C^{1,+}(\mathbb{R}^{N}),$ with support in
$\mathcal{R},$
\begin{equation}
\int_{\mathbb{R}^{N}}u_{p}(.,t)\xi^{q^{\prime}}dx+\frac{1}{2}\int_{0}^{t}%
\int_{\mathbb{R}^{N}}|\nabla u_{p}|^{q}\xi^{q^{\prime}}dx\leqq Ct\int%
_{\mathbb{R}^{N}}|\nabla\xi|^{q^{\prime}}dx+\int_{\mathbb{R}^{N}}%
\xi^{q^{\prime}}d\psi_{p}, \label{lac}%
\end{equation}%
\begin{equation}
\int_{\Omega}u_{p}(.,t)\xi dx+\int_{0}^{t}\int_{\Omega}(\nabla u_{p}.\nabla
\xi+\left\vert \nabla u_{p}\right\vert ^{q}\xi)dxdt=\int_{\Omega}\xi du_{0,p}.
\label{loc}%
\end{equation}

First suppose \textbf{ }$q<q_{\ast}$\textbf{. }From Theorem \ref{bapi},
$(\left\vert \nabla u_{p}\right\vert ^{q})$ is equi-integrable in $Q_{K,\tau}$
for any compact set $K\subset\mathcal{R}$ and $\tau\in\left(  0,T\right)  .$
From (\ref{loc}) for any $\zeta\in C_{c}(\mathcal{R}),$ for $p=p(\zeta)$ large
enough such that the support of $\zeta$ is contained in $\mathcal{R}%
_{1/p}^{int}\cap B_{p}$ (resp. $\mathcal{R}_{1/p}^{int}\cap\Omega_{1/p}^{int}%
$),
\[
\int_{\mathcal{R}}u_{p}(t,.)\zeta dx+\int_{0}^{t}\int_{\mathcal{R}}|\nabla
u_{p}|^{q}\zeta dx=-\int_{0}^{t}\int_{\mathcal{R}}\nabla u_{p}.\nabla\zeta
dx+\int_{\mathcal{R}}\zeta du_{0}.
\]
Then we can go to the limit as $p\rightarrow\infty$:
\[
\int_{\mathcal{R}}u(t,.)\zeta dx+\int_{0}^{t}\int_{\mathcal{R}}|\nabla
u|^{q}\zeta dx=-\int_{0}^{t}\int_{\mathcal{R}}\nabla u.\nabla\zeta
dx+\int_{\mathcal{R}}\zeta du_{0}.
\]
thus $\lim_{t\rightarrow0}\int_{\mathbb{R}^{N}}u(.,t)\zeta dx=\int%
_{\mathbb{R}^{N}}\zeta du_{0}.\medskip$

Next suppose \textbf{ }$q_{\ast}\leqq q\leqq2$\textbf{ }and\textbf{ }$u_{0}\in
L_{loc}^{1}\left(  \mathcal{R}\right)  ,$ or $q>2$ and $u_{0}$ is limit of a
sequence of nondecreasing continuous functions. Then $\psi_{p}\leqq u_{0}.$
From (\ref{lac}), we have \textbf{ }$\left\vert \nabla u\right\vert ^{q}\in
L_{loc}^{1}\left(  \left[  0,T\right)  ;L_{loc}^{1}(\mathcal{R})\right)  $
from the Fatou Lemma. Hence, from Lemma \ref{phi}, $u$ admits a trace $\mu
_{0}\in\mathcal{M(R}).$ For any fixed $\zeta\in C_{c}^{+}(\mathcal{R}),$ we
$\lim_{t\rightarrow0}\int_{\mathbb{R}^{N}}u_{p}(.,t)\zeta dx=\int%
_{\mathcal{R}}\zeta\psi_{p}dx.$ Since $(u_{p})$ is \textit{nondecreasing}, we
get
\[
\lim_{t\rightarrow0}\int_{\mathbb{R}^{N}}u(.,t)\zeta dx=\int_{\mathcal{R}%
}\zeta d\mu_{0}\geqq\lim_{t\rightarrow0}\int_{\mathbb{R}^{N}}u_{p}(.,t)\zeta
dx=\int_{\mathcal{R}}\zeta\psi_{p}dx.
\]
thus from the Beppo-Levy Theorem, $\mu_{0}\geqq u_{0}.$ Moreover for any
$\zeta\in C_{c}(\mathcal{R}),$ from (\ref{loc}),
\[
\int_{\mathcal{R}}u_{p}(t,.)\zeta dx+\int_{0}^{t}\int_{\mathcal{R}}|\nabla
u_{p}|^{q}\zeta dx=\int_{0}^{t}\int_{\mathcal{R}}u_{p}\Delta\zeta
dx+\int_{\mathcal{R}}\zeta\psi_{p}dx;
\]
and $\left(  u_{p}\right)  $ \textbf{ }is bounded in $L^{k}(Q_{K,\tau})$ for
any $k\in\left[  1,q_{\ast}\right)  ,$ for any compact set $K\subset
\mathcal{R},$ and $u_{p}\rightarrow u$ $a.e.$ in $\mathcal{R},$ then $(u_{p})$
converges strongly in $L^{1}(Q_{K,\tau})$, thus from the dominated convergence
Theorem and the Fatou Lemma,%
\[
\int_{\mathcal{R}}u(t,.)\zeta dx+\int_{0}^{t}\int_{\mathcal{R}}|\nabla
u|^{q}\zeta dx\leqq\int_{0}^{t}\int_{\mathcal{R}}u\Delta\zeta dx+\int%
_{\mathcal{R}}\zeta du_{0}.
\]
But from Lemma \ref{phi},
\[
\int_{\mathcal{R}}u(t,.)\zeta dx+\int_{0}^{t}\int_{\mathcal{R}}|\nabla
u|^{q}\zeta dx=\int_{0}^{t}\int_{\mathcal{R}}u\Delta\zeta dx+\int%
_{\mathcal{R}}\zeta d\mu_{0},
\]
then $\int_{\mathcal{R}}\zeta d\mu_{0}\leqq\int_{\mathcal{R}}\zeta du_{0},$
hence $\mu_{0}\leqq u_{0},$ hence $\mu_{0}=u_{0}.$

In any case $u$ admits the trace $(\mathcal{S},u_{0}).$
\end{proof}

\subsection{Solutions with any Borel measure}

In this part we consider the subcritical case with an arbitrary closed set
$\mathcal{S}.$

\begin{theorem}
\label{clos} Let $1<q<q_{\ast},$ and $\Omega=\mathbb{R}^{N}$ (resp. $\Omega$
bounded). Let $\mathcal{S}$ be a closed set in $\Omega,$ such that
$\mathcal{R}=\Omega\backslash\mathcal{S}$ is nonempty. Let $u_{0}%
\in\mathcal{M}^{+}\left(  \mathcal{R}\right)  $.\medskip

(i) Then there exists a solution $u$ of (\ref{un}) (resp. of $(D_{\Omega,T})$)
with initial trace $(\mathcal{S},u_{0})$, such that $u$ satisfies
(\ref{mini}), hence $u(t,.)$ converges to $\infty$ uniformly on $\mathcal{S}%
$.\medskip

(ii) There exists a minimal solution $u_{\min}$, satisfying the same
conditions.$\medskip$
\end{theorem}

\begin{proof}
Assume that $\Omega=\mathbb{R}^{N}$ (resp. $\Omega$ bounded) (i) Existence of
a solution. Let $B(x_{0},\eta)\subset\Omega\backslash\mathcal{S},$ and
$\delta_{0}$ small enough such that $B(x_{0},\eta)\subset\Omega\backslash
\mathcal{S}_{\delta_{0}}^{ext}.$ For any $\delta\in\left(  0,\delta
_{0}\right)  $ we can suppose that $\mathcal{S}_{\delta}^{ext}=\overline
{\omega_{\delta}}\cap\Omega,$ where $\omega_{\delta}$ is a smooth open subset
of $\Omega$ (if $\mathcal{S}_{\delta}^{ext}$ is not smoothenough regular, we
replace it by a smooth open set $\mathcal{S}_{\delta}^{\prime ext}$such that
$\mathcal{S}\subset\mathcal{S}_{\delta}^{\prime ext}\subset\mathcal{S}%
_{\delta}^{ext})$. Let $u_{\delta}$ be the solution with initial trace
$(\mathcal{S}_{\delta}^{ext},u_{0}\llcorner\mathcal{(}\Omega\backslash
\mathcal{S}_{\delta}^{ext}))$ constructed at Theorem \ref{openex}. Then
$u_{\delta}$ admits the trace $u_{0}$ on $B(x_{0},\eta),$ thus it also
satisfies the estimates (\ref{pluc}) (resp. (\ref{wo})), thus $(u_{\delta
})_{\delta<\delta_{0}}$ is uniformly locally bounded in $Q_{\Omega,T}.$ From
Theorem \ref{raploc} (resp. \ref{regdir}), one can extract a subsequence
converging in $C_{loc}^{2,1}(Q_{\Omega,T})$ to a solution $u$ of (\ref{un}) in
$Q_{\mathbb{R}^{N},T}$ (resp. of $(D_{\Omega,T})$). As in the proof of Theorem
\ref{openex}, for any compact $K\subset\mathcal{R},$ taking $\delta<\delta
_{K}$ small enough so that $K\subset\Omega\backslash\mathcal{S}_{\delta_{K}%
}^{ext}$, and choosing a test function $\xi$ with compact support in $K$ in
$\mathcal{R},$ we obtain that ($\left\vert \nabla u_{\delta}\right\vert
^{q})_{\delta<\delta_{K}}$ is equi-integrable in $Q_{K,\tau}$ for any $\tau
\in(0,T).$ Then we get for any $\xi\in C_{c}(\mathcal{R}),$
\[
\int_{\mathbb{R}^{N}}u(t,.)\xi dx+\int_{0}^{t}\int_{\mathbb{R}^{N}}|\nabla
u|^{q}\xi dx=-\int_{0}^{t}\int_{\mathbb{R}^{N}}\nabla u.\nabla\xi
dx+\int_{\mathbb{R}^{N}}\xi du_{0}.
\]
thus $\lim_{t\rightarrow0}\int_{\mathbb{R}^{N}}u(.,t)\xi dx=\int%
_{\mathbb{R}^{N}}\xi du_{0}.$ Moreover for any $x_{0}\in\mathcal{S},$
$u_{\delta}\geqq$ $Y_{\left\{  x_{0}\right\}  }$ in $Q_{\mathbb{R}^{N},T},$
(resp. $u_{\delta}\geqq$ $Y_{\left\{  x_{0}\right\}  }^{\Omega}$ in
$Q_{\Omega,T})$ from Proposition \ref{mv}, hence the same happens for $u,$
which implies (\ref{ddd}). Thus $u$ admits $(\mathcal{S},u_{0})$ as initial
trace, and $u(.,t)$ converges uniformly on $\mathcal{S}$ to $\infty$ as
$t\rightarrow0$.\medskip

\noindent(ii) Existence of a \textbf{minimal} solution. \medskip

Assume that $\Omega=\mathbb{R}^{N}.$ Let $A$ be the set of solutions with
initial trace $(\mathcal{S},u_{0}).$ We consider for fixed $\epsilon>0,$ the
Dirichlet problem in $Q_{B_{p},T},$ $p\geqq1,$ with initial data
$m(x,\epsilon)=\inf_{v\in A}v(x,\epsilon).$ Thus $0\leqq m(x,\epsilon)\leqq
u(x,\epsilon),$ where $u$ has been defined at step (i), and $u\in
C^{2,1}(Q_{\mathbb{R}^{N},T}),$ thus $m(.,\epsilon)\in L_{loc}^{1}\left(
\mathbb{R}^{N}\right)  .$ Since $m\in L^{1}(B_{p}),$ there exists a unique
solution $w_{p,\epsilon}$ of $(D_{B_{p},T})$ with initial data $m(x,\epsilon)$
in $B_{p}.$ From Corollary \ref{cpri}, $w_{p,\epsilon}(x,t)\leqq
v(x,t+\epsilon)$ for any $v\in A$ and $x\in B_{p}.$ Moreover for any $v\in A$
and any $x_{0}\in\mathcal{S},$ there holds $v\geqq Y_{\left\{  x_{0}\right\}
}\geqq Y_{x_{0}}^{B_{p}},$ thus $m(x,\epsilon)\geqq Y_{x_{0}}^{B_{p}%
}(x,\epsilon),$ hence $w_{p,\epsilon}(x,t)\geqq Y_{x_{0}}^{B_{p}}%
(x,t+\epsilon)$, from \cite[Proposition 2.1]{SoZh}. For any $z_{0}\in B_{p}$
and $\gamma>0$ such that $\mathcal{B}=B(z_{0},\gamma)$ satisfies
$\overline{\mathcal{B}}\subset\mathcal{R}\cap B_{p},$ let $w_{U}$ be the
unique solution of the Dirichlet problem in $\mathcal{B}$ with initial data
$u_{0}\llcorner\mathcal{B}.$ Then from Corollary \ref{cpri}, $v(x,t)\geqq
w_{\mathcal{B}}(x,t)$ in $Q_{\mathcal{B},T},$ for any $v\in A,$ thus
$m(x,\epsilon)\geqq w_{\mathcal{B}}(x,\epsilon),$ thus $w_{p,\epsilon
}(x,t)\geqq w_{\mathcal{B}}(x,t+\epsilon).$

Next we go to the limit as $\epsilon\rightarrow0.$ From Theorem \ref{regdir},
one can extract a subsequence, still denoted $\left(  w_{p,\epsilon}\right)
,$ converging $a.e.$ to a solution $w_{p}$ of the Dirichlet problem
$(D_{B_{p},T}).$ And in $B_{p}$ (with the notations above), $w_{p}\leqq v$ for
any $v\in A,w_{p}\geqq Y_{x_{0}}^{B_{p}}$ and $w_{p}\geqq w_{U}$. Finally we
go to the limit as $p\rightarrow\infty.$ Since $u$ is locally bounded, then
$(w_{p})$ is uniformly locally bounded. From Theorem \ref{raploc}, one can
extract a subsequence converging in $C_{loc}^{2,1}(Q_{\mathbb{R}^{N},T})$ to a
weak solution denoted $u_{\min}$ of (\ref{un}) in $Q_{\mathbb{R}^{N},T}$. Then
$u_{\min}$ satisfies $u_{\min}\leqq v$ for any $v\in A,$ and $u_{\min}\geqq
Y_{x_{0}}^{B_{p}}$ for any $x_{0}\in\mathcal{S},$ and $u_{\min}\geqq w_{U}$
for any $z_{0}\in\mathcal{R}$ and $\gamma>0$ such that $\mathcal{B}%
=B(z_{0},\gamma)$ satisfies $\overline{\mathcal{B}}\subset\mathcal{R}.$ As a
consequence $u_{\min}$ satisfies the trace condition (\ref{ddd}) on
$\mathcal{S}.$ And for any $z_{0}\in\mathcal{R}$, and any $\xi\in C_{c}%
^{0}(\mathcal{R})$ with support in $U,$%
\[
\int_{\mathcal{R}}u(.,t)\xi dx\geqq\int_{\mathcal{R}}u_{\min}(.,t)\xi
dx\geqq\int_{\mathcal{R}}w_{U}(.,t)\xi dx
\]
hence
\[
\lim_{t\rightarrow0}\int_{\mathcal{R}}u_{\min}(.,t)\xi dx=\int_{\mathcal{R}%
}\xi du_{0}.
\]
Then $u_{\min}$ admits the trace $(\mathcal{S},u_{0}).$ Thus $u_{\min}$ is
minimal, and $u_{\min}=\min_{v\in A}v.$\medskip

Assume that $\Omega$ is bounded. The proof still works with $B_{p}$ replaced
by $\Omega,$ which requires only to go to the limit in $\varepsilon$ and use
Theorem \ref{regdir}.\medskip
\end{proof}

In the case where $u_{0}$ is a \textit{bounded} measure we can give more
convergence results:

\begin{corollary}
\label{faib}Under the assumptions of Theorem \ref{clos} suppose that $u_{0}%
\in\mathcal{M}_{b}^{+}\left(  \mathcal{R}\right)  .$ Then for any $\varphi\in
C_{b}(\Omega)$ with support in $\mathcal{R}$, $u(.,t)\varphi\in L^{1}\left(
\mathcal{R}\right)  $ for any $t\in\left(  0,T\right)  ,$ and
\begin{equation}
\lim_{t\rightarrow0}\int_{\mathcal{R}}u(.,t)\varphi dx=\int_{\mathcal{R}%
}\varphi du_{0}, \label{blo}%
\end{equation}
and similarly for $u_{\min}.$ More precisely, if $\Omega=\mathbb{R}^{N},$
(\ref{blo}) is valid for any weak solution $v$ of (\ref{un}) with trace
$(\mathcal{S},u_{0})$.
\end{corollary}

\begin{proof}
First assume that $\Omega=\mathbb{R}^{N}$ and $v$ is any weak solution with
trace $(\mathcal{S},u_{0})$ let $\psi\in C_{b}^{1}(\mathbb{R}^{N})$\textbf{
}with support in $\mathcal{R}$, and $\varphi_{n}\in\mathcal{D}\left(
\mathbb{R}^{N}\right)  $ with values in $\left[  0,1\right]  ,$ with
$\varphi_{n}=1$ on $B_{n},$ $0$ on $B_{2n},$ and $(\left\vert \nabla
\varphi_{n}\right\vert )$ bounded. Then from (\ref{zif}),
\[
\int_{\mathcal{R}}v(.,t)(\psi\varphi_{n})^{q^{\prime}}dx\leqq C(q)t\int%
_{\mathbb{R}^{N}}|\nabla(\psi\varphi_{n})|^{q^{\prime}}dx+\int_{\mathbb{R}%
^{N}}(\psi\varphi_{n})^{q^{\prime}}du_{0}\leqq Ct+\int_{\mathcal{R}}%
\psi^{q^{\prime}}du_{0};
\]
thus $v(.,t)\psi^{q^{\prime}}\in L^{1}\left(  \mathcal{R}\right)  ,$ and
$\lim\sup_{t\rightarrow0}$ $\int_{\mathcal{R}}v(.,t)\psi^{q^{\prime}}%
dx\leqq\int_{\Omega}\psi^{q^{\prime}}du_{0}$ from the Fatou Lemma. And
\[
\lim\inf_{t\rightarrow0}\int_{\mathcal{R}}v(.,t)\psi^{q^{\prime}}dx\geqq
\lim_{t\rightarrow0}\int_{\mathcal{R}}v(.,t)(\psi\varphi_{n})^{q^{\prime}%
}dx=\int_{\mathcal{R}}(\psi\varphi_{n})^{q^{\prime}}du_{0},
\]
thus from the Beppo-Levy Theorem, we get (\ref{blo}) by density.

Next suppose that $\Omega$ is bounded, note that $u$ can be obtained as a
limit in $C_{loc}^{2,1}(Q_{\Omega,T})\cap C_{loc}^{1,0}\left(  \overline
{\Omega}\times\left(  0,T\right)  \right)  $ of classical solutions $u_{n}$
with smooth data $u_{n,0}=u_{n,0}^{1}+u_{n,0}^{2}$ with supp$u_{n,0}%
^{1}\subset\overset{\circ}{\mathcal{S}_{3\delta_{0}}^{ext}}$, supp$u_{n,0}%
^{1}\subset\mathcal{R}$, and $(u_{n,0}^{1})$ converges to $u_{0}$ weakly in
$\mathcal{M}_{b}(\mathcal{R)}$. For any nonnegative $\xi\in C_{b}^{1}(\Omega)$
with support in $\mathcal{R}$,
\[
\int_{\mathcal{R}}u_{n}(.,t)\xi^{q^{\prime}}dx\leqq C(q)t\int_{\mathcal{R}%
}|\nabla\xi|^{q^{\prime}}dx+\int_{\Omega}\xi^{q^{\prime}}u_{n,0}^{2}dx,
\]
from Remark \ref{norm}, hence
\[
\int_{\mathcal{R}}u(.,t)\xi^{q^{\prime}}dx\leqq C(q)t\int_{\mathcal{R}}%
|\nabla\xi|^{q^{\prime}}dx+\int_{\Omega}\xi^{q^{\prime}}du_{0},
\]
and then $\lim\sup_{t\rightarrow0}$ $\int_{\mathcal{R}}u(.,t)\psi^{q^{\prime}%
}dx\leqq\int_{\Omega}\psi^{q^{\prime}}du_{0}.$ And for any $\varphi_{n}%
\in\mathcal{D}\left(  \Omega\right)  $ with values in $\left[  0,1\right]  ,$
with $\varphi_{n}=1$ on $\mathcal{R}_{1/n}^{int},$
\[
\lim\inf_{t\rightarrow0}\int_{\mathcal{R}}u(.,t)\psi^{q^{\prime}}dx\geqq
\lim_{t\rightarrow0}\int_{\mathcal{R}}v(.,t)(\psi\varphi_{n})^{q^{\prime}%
}dx=\int_{\mathcal{R}}(\psi\varphi_{n})^{q^{\prime}}du_{0},
\]
Thus $u$ still satisfies (\ref{blo}). The same happens for $u_{\min},$ since
$\lim\sup_{t\rightarrow0}\int_{\mathcal{R}}u_{\min}(.,t)\varphi dx\leqq
\int_{\mathcal{R}}\varphi du_{0}$ and $\lim\inf_{t\rightarrow0}\int%
_{\mathcal{R}}u_{\min}(.,t)\psi^{q^{\prime}}dx\geqq\int_{\mathcal{R}}%
(\psi\varphi_{n})^{q^{\prime}}du_{0}.$
\end{proof}

\begin{remark}
Assume $1<q<q_{\ast}.$ Note some consequences of Theorems \ref{clos} and
\ref{openex}.

(i) For any constant $C>0,$ there exists a minimal solution $u_{C}$ with trace
$(\left\{  0\right\}  ,C\left\vert x\right\vert ^{-a})).$ Then $u_{C}$ is
radial and self-similar. This shows again the existence of the solutions of
example 2, Section \ref{2}. This shows that the set $\left\{  C(\beta
):\beta>F(0)\right\}  ,$ where $F$ and $C(\beta)$ ere defiend at (\ref{selfs})
and (\ref{cbeta}), is equal to $\left(  0,\infty\right)  .\medskip$

(ii) Suppose $N=1.$ For any $C>0$ there exists a minimal solution
$\widetilde{u}_{C}$ with trace $(\left[  0,\infty\right)  ,C(x^{-})^{-a})$; it
is self-similar, $\widetilde{u}_{C}(x,t)=t^{-a/2}\widetilde{f}(t^{-1/2}x);$ as
in the proof of Proposition \ref{selfdys}, we obtain that $\widetilde{f}$ is
increasing and lim$_{\eta\rightarrow\infty}\widetilde{f}(\eta)\eta
^{-q^{\prime}}=c,$ and $\lim_{\eta\rightarrow-\infty}\eta\widetilde{f}%
^{\prime}(\eta)/\widetilde{f}(\eta)=a,$ and then lim$_{\eta\rightarrow-\infty
}\widetilde{f}(\eta)\left\vert \eta\right\vert ^{a}=C$. In the same way, for
any $C>0,$ there exists a minimal solution $\widehat{u}_{C}$ with trace
$(\left\{  0\right\}  ,C(x^{+})^{-a})$; then it is self-similar,
$\widehat{u}_{C}(x,t)=t^{-a/2}\widehat{f}(t^{-1/2}x),$ where $\eta
\longmapsto\widehat{f}(\eta)$ is defined on $\mathbb{R},$ and we check that
$\widehat{f}$ has an exponential decay at $-\infty,$ and lim$_{\eta
\rightarrow\infty}\widehat{f}(\eta)\eta^{a}=C.$ \medskip
\end{remark}

Next we look for a maximal solution when the measure $u_{0}$ is bounded. A
crucial point in case $\Omega=\mathbb{R}^{N}$ is the obtention of an upper
estimate, based on Theorems \ref{local} and \ref{fund}:

\begin{proposition}
\label{fdm}$1<q\leqq2.$ Let $\mathcal{S}$ be a compact set in $\mathbb{R}^{N}%
$, and $u_{0}\in\mathcal{M}^{+}\left(  \mathbb{R}^{N}\backslash\mathcal{S}%
\right)  $, bounded at $\infty.$ Then any weak solution $v$ of (\ref{un}) in
$Q_{\mathbb{R}^{N},T}$ with trace $(\mathcal{S},u_{0})$ satisfies, for any
$0<\epsilon<\tau<T,$
\begin{equation}
\left\Vert v\right\Vert _{L^{\infty}((\epsilon,\tau);L^{\infty}(\mathbb{R}%
^{N}))}\leqq C,\qquad C=C(N,q,\epsilon,\tau). \label{fet}%
\end{equation}

\end{proposition}

\begin{proof}
Let $\tau\in0,T)$. We take $\eta=1$ and $x_{0}\in\mathbb{R}^{N}\backslash
\mathcal{S}_{1}$ in (\ref{pluc}). Then for any $(x,t)\in Q_{\mathbb{R}%
^{N},\tau},$
\begin{equation}
v(x,t)\leqq C(q)t^{-\frac{1}{q-1}}\left\vert x-x_{0}\right\vert ^{q^{\prime}%
}+C(N,q)(t^{-\frac{1}{q-1}}+t+\int_{B(x_{0},1)}du_{0}). \label{esi}%
\end{equation}
In particular it holds in $\mathcal{S}_{2}\times\left(  0,\tau\right]  $. And
for any $(x,t)\in\mathbb{R}^{N}\backslash\mathcal{S}_{2},$ since $u_{0}%
\in\mathcal{M}_{b}^{+}\left(  \mathbb{R}^{N}\backslash\mathcal{S}_{1}\right)
$, from (\ref{locma}),
\begin{equation}
v(x,t)\leqq C(N,q,\tau)t^{-N/2}(t+\int_{B(x_{0},1)}du_{0})\leqq C(N,q,\tau
)t^{-N/2}(t+\int_{\mathbb{R}^{N}\backslash\mathcal{S}_{1}}du_{0}). \label{esu}%
\end{equation}
Then (\ref{fet}) follows.\medskip
\end{proof}

\begin{theorem}
\label{soma}Let $1<q<q_{\ast}.$ Let $\Omega=\mathbb{R}^{N}$ (resp. $\Omega$
bounded). Assume that $\mathcal{S}$ is compact in $\Omega$ and $u_{0}%
\in\mathcal{M}_{b}^{+}\left(  \Omega\right)  $ with support in $\mathcal{R}%
\cup\overline{\Omega},$ where $\mathcal{R}=\Omega\backslash\mathcal{S}%
$.\medskip\ Then there exists a maximal solution $u$ of (\ref{un}) (resp. of
$(D_{\Omega,T})$) among the solutions with trace $(\mathcal{S},u_{0})$ (resp.
among the solutions $v$ of trace $(\mathcal{S},u_{0})$ such that $v(.,t)$
converges weakly in $\mathcal{R}$ to $u_{0}$ as $t\rightarrow0).$
\end{theorem}

\begin{proof}
Assume $\Omega=\mathbb{R}^{N}$ (resp. $\Omega$ bounded). Let $\delta>0$ be
fixed, such that $\delta<d(\mathcal{S},$supp$u_{0})/3,$ hence supp$u_{0}%
\subset\Omega\backslash\mathcal{S}_{3\delta}.$ Let $u_{\delta}$ be the
solution with initial trace $(\mathcal{S}_{\delta}^{ext},u_{0})$ constructed
at Theorem \ref{openex}.

Let $v$ be any weak solution with trace $(\mathcal{S},u_{0})$ (resp. and such
that $v(.,t)$ converges weakly in $\mathcal{M}_{b}(\mathcal{R}$). Then
$v(.,t)\leqq C(N,q,\delta)t$ in $\mathcal{K}_{\delta}=\mathcal{S}_{5\delta
/2}^{ext}\backslash\overset{\circ}{\mathcal{S}_{\delta/2}^{ext}},$ from Lemma
\ref{bord} (resp. from (\ref{iss}) in $\mathcal{O}=\overset{\circ
}{\mathcal{S}_{3\delta}^{ext}}\backslash\mathcal{S}_{\delta}^{ext}$, valid
since $v\in C(\left[  0,T\right)  \times\mathcal{O)}$). Let $\epsilon_{0}>0.$
Then there exists $\tau_{0}=\tau_{0}(\epsilon_{0},\delta)<T$ such that
$v(.,t)\leqq\epsilon_{0}$ in $\mathcal{K}_{\delta}\times\left(  0,\tau
_{0}\right]  .$ Let $\epsilon<\tau_{0},$ and $C_{\epsilon}=\max_{\mathcal{S}%
_{2\delta}}v(.,\epsilon).$ Since $u_{\delta}$ converges to $\infty$ uniformly
on the compact sets of $\mathcal{S}_{\delta}^{ext},$ there exists
$\tau_{\epsilon}<\tau_{0}$ such that for any $\theta\in\left(  0,\tau
_{\epsilon}\right)  ,$ $u_{\delta}(.,\theta)\geqq C_{\epsilon}\geqq
v(.,\epsilon)$ in $\mathcal{S}_{\delta/2}.$ Since $v(.,\epsilon)\leqq
\epsilon_{0}$ in $\mathcal{K}_{\delta},$ there holds $v(.,\epsilon)\leqq
u_{\delta}(.,\theta)+\epsilon_{0}$ in $\mathcal{S}_{2\delta}.$ And
$v(.,t)\leqq\epsilon_{0}$ on $\partial\mathcal{S}_{2\delta}\times\left(
0,\tau_{0}\right]  ,$ thus $v(.,t+\epsilon)\leqq u_{\delta}(.,t+\theta
)+\epsilon_{0}$ in $\mathcal{S}_{2\delta}\times\left(  0,\tau_{0}%
-\epsilon\right]  $ from the comparison principle. As $\theta\rightarrow0,$
then $\epsilon\rightarrow0,$ we get
\begin{equation}
v(.,t)\leqq u_{\delta}(.,t)+\epsilon_{0}\text{ \quad in }\mathcal{S}_{2\delta
}\times\left(  0,\tau_{0}\right]  . \label{hyc}%
\end{equation}
Otherwise, since $u_{0}\in\mathcal{M}_{b}^{+}\left(  \Omega\right)  ,$ there
exists a unique solution $w$ of $(P_{\Omega,T})$ with initial data $u_{0},$
from Theorem \ref{souc}. We claim that
\begin{equation}
v(x,t)\leqq w(x,t)+\epsilon_{0},\text{\qquad in }\overline{\Omega
\backslash\mathcal{S}_{2\delta}}\times\left(  0,\tau_{0}\right]  . \label{mai}%
\end{equation}

Indeed let $\varphi_{\delta}\in C(\overline{\Omega})$ with values in $\left[
0,1\right]  $ with support in $\overline{\Omega}\backslash\mathcal{S}%
_{2\delta}$ and $\varphi_{\delta}=1$ on $\mathbb{R}^{N}\backslash
\mathcal{S}_{5\delta/2}.$ From Proposition \ref{fdm} (resp. from Theorem
\ref{souc}), the function $x\longmapsto v(x,\tau_{0}/n)$ is bounded, and
continuous. Let $w_{\delta,n}$ be the solution of (\ref{un}) in $Q_{\Omega,T}$
with initial data $v(.,\tau_{0}/n)\varphi_{\delta}.$ As $n\rightarrow\infty,$
$v(.,\tau_{0}/n)\varphi_{\delta}$ converges to $u_{0}\varphi_{\delta}=u_{0}$
weakly in $\mathcal{M}_{b}\left(  \mathbb{R}^{N}\right)  $, from Remark
\ref{faib} (resp. from our assumption). Hence $w_{\delta,n}$ converges to $w,$
from Proposition \ref{cpro}. And then
\[
v(.,\tau_{0}/n)=v(.,\tau_{0}/n)\varphi_{\delta}+v(.,\tau_{0}/n)(1-\varphi
_{\delta})\leqq w_{\delta,n}(.,0)+\epsilon_{0}%
\]
in $\overline{\Omega\backslash\mathcal{S}_{2\delta}},$ and on the lateral
boundary of $\overline{\Omega\backslash\mathcal{S}_{2\delta}}\times\left(
0,\tau_{0}(1-1/n)\right]  ,$ there holds $v(x,t+\tau_{0}/n)\leqq\epsilon_{0}.$
Then $v(x,t+\tau_{0}/n)\leqq w_{\delta,n}(.,t)+\epsilon_{0}$ in $\overline
{\Omega\backslash\mathcal{S}_{2\delta}}\times\left(  0,\tau_{0}(1-1/n)\right]
.$ As $n\rightarrow\infty,$ we deduce (\ref{mai}).

Next we get easily that $w\leqq$ $u_{\delta}$ on $\overline{\Omega
\backslash\mathcal{S}_{2\delta}}\times\left(  0,\tau_{0}\right]  ,$ by
considering their approximations, hence
\begin{equation}
v(x,t)\leqq u_{\delta}(x,t)+\epsilon_{0},\text{\qquad in }\overline
{\Omega\backslash\mathcal{S}_{2\delta}}\times\left(  0,\tau_{0}\right]  .
\label{hac}%
\end{equation}
As a consequence, from (\ref{hyc}) and \ref{hac}),
\[
v(x,t)\leqq u_{\delta}(x,t)+\epsilon_{0},\text{\qquad in }\overline{\Omega
}\times\left(  0,\tau_{0}\right]  .
\]
The last step is to prove that the inequality holds up to time $T.$ We can
apply the comparison principle because, from\textbf{ }Proposition \ref{fdm},
$u$ and $v\in C_{b}((\epsilon,T);C_{b}(\mathbb{R}^{N})$ for any $\epsilon>0$
(resp. because $v$ and $u_{\delta}$ are classical solutions of $(D_{\Omega
,T})).$ Then
\[
v(x,t)\leqq u_{\delta}(x,t)+\epsilon_{0},\text{\qquad in }\overline{\Omega
}\times(0,T)
\]
As $\epsilon_{0}\rightarrow0,$ we deduce that $v\leqq u_{\delta}.$ Finally as
$\delta\rightarrow0,$ up to a subsequence, $\left\{  u_{\delta}\right\}  $
converges to a solution $u$ of (\ref{un}) (resp. of $(D_{\Omega,T})$, such
that $v\leqq u$, thus $u$ satisfies (\ref{ddd}). As in Theorem \ref{openex},
by integrability of $(\left\vert \nabla u_{\delta}\right\vert ^{q})$ we obtain
that $u$ admits the trace $u_{0}$ in $\mathcal{R},$ thus $u$ has the trace
$(\mathcal{S},u_{0})$ (resp. and the convergence holds weakly in
$\mathcal{M}_{b}(\mathcal{R)}$). Thus $u$ is maximal.\medskip
\end{proof}

From Theorems \ref{clos} and \ref{soma}, this ends the proof of Theorem
\ref{ess}.

\section{The case $0<q\leqq1$\label{5}}

Notice that Theorem \ref{gul} is also valid for $q=1.$ In fact it can be
improved when $q$ is subcritical, and extended to the case $q<1.$

\begin{theorem}
\label{gul2}(i) Let $0<q<q,$ and $\Omega$ be any domain in $\mathbb{R}^{N}$.
Let $u$ be any (signed) weak solution of (\ref{un}) in $Q_{\Omega,T}.$ Then
$u\in C_{loc}^{2+\gamma,1+\gamma/2}(Q_{\Omega,T})$ for some $\gamma\in\left(
0,1\right)  .$ If $\Omega$ is bounded, any weak solution $u$ of problem
$(D_{\Omega,T})$ satisfies $u\in C^{1,0}\left(  \overline{\Omega}\times\left(
0,T\right)  \right)  \cap C_{loc}^{2+\gamma,1+\gamma/2}(Q_{\Omega,T})$ for
some $\gamma\in\left(  0,1\right)  .$\medskip

(ii) Let $0<q\leqq1$ and $\Omega$ bounded. For any sequence of weak
nonnegative solutions $\left(  u_{n}\right)  $ of $(D_{\Omega,T}),$ bounded in
$L_{loc}^{\infty}((0,T);L^{1}(\Omega))$ one can extract a subsequence
converging in $C_{loc}^{2,1}(Q_{\Omega,T})\cap C^{1,0}\left(  \overline
{\Omega}\times\left(  0,T\right)  \right)  $ to a weak solution $u$ of
$(D_{\Omega,T})$.
\end{theorem}

\begin{proof}
(i) From our assumptions, $u\in C((0,T);L_{loc}^{1}(Q_{\Omega,T})),$ thus
$u\in L_{loc}^{\infty}((0,T);L_{loc}^{1}(Q_{\Omega,T})).$ We can write
(\ref{un}) under the form $u_{t}-\Delta u=f,$ with $f=$ $-|\nabla u|^{q}.$
From Theorem \ref{bapi} $u\in L_{loc}^{1}((0,T);W_{loc}^{1,k}(\Omega)$ for any
$k\in\left[  1,q_{\ast}\right)  $ and satisfies (\ref{pi}).\medskip

First suppose $q\leqq1.$ We choose $k\in(1,q_{\ast})$, thus $(|\nabla
u|+\left\vert u\right\vert )\in L_{loc}^{k}\left(  Q_{\Omega,T}\right)  .$
Then $u\in\mathcal{W}_{loc}^{2,1,k}(Q_{\Omega,T}),$ see \cite[ theorem
IV.$9.1$]{LSU}. From the Gagliardo-Nirenberg inequality, for almost any
$t\in(0,T)$,
\[
\Vert\nabla u(.,t)\Vert_{L^{kq_{\ast}}(\omega)}\leqq c\Vert u(t)\Vert
_{W^{2,k}(\omega)}^{\frac{1}{q_{\ast}}}\Vert u(t)\Vert_{L^{1}(\omega
)}^{1-\frac{1}{q_{\ast}}},
\]
where $c=c(N,s,\omega)$. Hence we obtain $|\nabla u|\in L_{loc}^{kq_{\ast}%
}\left(  \Omega\right)  .$ In the same way
\[
\Vert u(.,t)\Vert_{L^{kq_{\ast}}(\omega)}\leqq c\Vert u(t)\Vert_{W^{2,s}%
(\omega)}^{\theta}\Vert u(t)\Vert_{L^{1}(\omega)}^{1-\theta},
\]
with $\theta=(1-1/kq_{\ast})/((N+2)/N-1/s)<1.$ Therefore $|u|\in
L_{loc}^{sq_{\ast}}\left(  \Omega\right)  .$ Then $u\in\mathcal{W}%
_{loc}^{2,1,kq_{\ast}}(Q_{\Omega,T}).$ By induction $u\in\mathcal{W}%
_{loc}^{2,1,k(q_{\ast})^{n}}(Q_{\Omega,T})$ for any $n\geqq1.$ Choosing $n$
such that $k(q_{\ast})^{n}>N+2,$ we deduce that $|\nabla u|\in C^{\delta
,\delta/2}(Q_{\omega,s,\tau})$ for any $\delta\in(0,1-(N+2)/s(q_{\ast})^{n}),$
see \cite[Lemma II.3.3]{LSU}. Then $f\in C_{loc}^{\delta q,\delta
q/2}(Q_{\Omega,T})$, thus $u\in C^{2+\delta q,1+\delta q/2}(Q_{\omega,s,\tau
})$.\medskip

Next suppose $1<q<q_{\ast}.$ we choose $k\in(1,q_{\ast}/q)$, hence $(|\nabla
u|^{q}+\left\vert u\right\vert )\in L_{loc}^{k}\left(  \Omega\right)  ;$ as
above, $|\nabla u|+\left\vert u\right\vert \in L_{loc}^{kq_{\ast}}\left(
\Omega\right)  ,$ hence $(|\nabla u|^{q}+\left\vert u\right\vert )\in
L_{loc}^{kq_{\ast}/q}\left(  \Omega\right)  ;$ then $u\in\mathcal{W}%
_{loc}^{2,1,kq_{\ast}/q}(Q_{\Omega,T}).$ By induction we get again that
$|\nabla u|\in C_{loc}^{\delta,\delta/2}(Q_{Q,T})$ for some $\delta\in(0,1),$
then $f\in C_{loc}^{\gamma,\gamma/2}(Q_{Q,T})$ for some $\gamma\in(0,1),$ thus
$u\in C_{loc}^{2+\gamma,1+\gamma/2}(Q_{\Omega,T})$ for some $\gamma\in\left(
0,1\right)  .$\medskip

If $\Omega$ is bounded, and $u$ is a weak solution of $(D_{\Omega,T}),$ then
$u$ satisfies (\ref{pa}). In the same way, $u\in\mathcal{W}^{2,1,k}%
(Q_{\Omega,s,\tau})$, and by induction $u\in C^{1,0}\left(  \overline{\Omega
}\times\left(  0,T\right)  \right)  \cap C_{loc}^{2+\gamma,1+\gamma
/2}(Q_{\Omega,T}).$\medskip

(ii) From (\ref{pa}), $\left\Vert u\right\Vert _{C^{1,0}(\overline
{Q_{\Omega,s,\tau}})}+\left\Vert \nabla u\right\Vert _{C^{\gamma,\gamma
/2}(\overline{Q_{\Omega,s,\tau})}}$ is bounded in terms of $\left\Vert |\nabla
u|^{q}\right\Vert _{L^{1}(Q_{\omega,s,\tau})}+\left\Vert u(.,s)\right\Vert
_{L^{1}(\Omega)}.$ And since $u$ is nonnegative, from \cite[lemma 5.3]{BiDao2}
(valid for $q>0),$
\begin{equation}
\int_{\Omega}u(t,.)dx+\int_{s}^{t}\int_{\Omega}|\nabla u|^{q}dx\leqq
\int_{\Omega}u(s,.)dx. \label{lom}%
\end{equation}
Thus $\left\Vert |\nabla u|^{q}\right\Vert _{L^{1}(Q_{\omega,s,\tau})}$ is
bounded in terms of $\left\Vert u(.,s)\right\Vert _{L^{1}(\Omega)}.$ Then one
can extract a subsequence converging in $C_{loc}^{2,1}(Q_{\Omega,T})\cap
C^{1,0}\left(  \overline{\Omega}\times\left(  0,T\right)  \right)  $ to a weak
solution $u$ of $(D_{\Omega,T})$.
\end{proof}

\begin{remark}
In case of the Dirichlet problem, the result also follows from \cite[Theorem
3.2 and Proposition 5.1]{BeDa}, by using the uniqueness of the solution in
$(Q_{\omega,\epsilon,T}).$
\end{remark}

Next we prove the uniqueness result of Theorem \ref{exun}. For that purpose we
recall a comparison property given in \cite[Lemma 4.1]{Al}:

\begin{lemma}
[\cite{Al}]\label{Alaa} Let $\Omega$ be bounded, and $A\in L^{\sigma
}(Q_{\Omega,T})$ with $\sigma>N+2.$ Let $w\in L^{1}((0,T);W_{0}^{1,1}%
(\Omega)),$ with $w\in C(\left(  0,T\right]  ;L^{1}(\Omega),$ such that
$w_{t}-\Delta w\in L^{1}(Q_{\Omega,T}),$ and $w(.,t)$ converges to a
nonpositive measure $w_{0}\in\mathcal{M}_{b}(\Omega)$, weakly in
$\mathcal{M}_{b}(\Omega),$ and
\[
w_{t}-\Delta w\leqq A.\nabla w\qquad\text{in }\mathcal{D}^{\prime}%
(Q_{\Omega,T}).
\]
Then $w\leqq0$ in $Q_{\Omega,T}.$
\end{lemma}

\begin{proof}
[Proof of Theorem \ref{exun}]From \cite{BeDa}, the problems with initial data
$u_{0},v_{0}$ admit at least two solutions $u,v.$ Then $f=|\nabla u|^{q}\in
L_{loc}^{1}(\left[  0,T\right)  ;L^{1}\left(  \Omega\right)  ).$ And by
hypothesis $u\in C((0,T);L^{1}\left(  \Omega\right)  )\cap L^{1}%
((0,T);W_{0}^{1,1}\left(  \Omega\right)  ).$ Assume that $u_{0}\leqq v_{0}.$
Let $w=u-v.$ Then we have $w\in C((0,T);L^{1}\left(  \Omega\right)  )\cap
L^{1}((0,T);W_{0}^{1,1}\left(  \Omega\right)  )$, $\left\vert \nabla
w\right\vert \in L^{k}(Q_{\Omega,\tau})$ for any $k\in\left[  1,q_{\ast
}\right)  $ and $\tau\in(0,T).$ Setting $g=|\nabla u|^{q}-|\nabla v|^{q},$
then $w$ is the unique solution of the problem%
\[
\left\{
\begin{array}
[c]{l}%
w_{t}-\Delta w=g,\quad\text{in}\hspace{0.05in}Q_{\Omega,T},\\
w=0,\quad\text{on}\hspace{0.05in}\partial\Omega\times(0,T),\\
\lim_{t\rightarrow0}w(.,t)=u_{0}-v_{0},\text{ weakly in }\mathcal{M}%
_{b}(\Omega).
\end{array}
\right.
\]
Since $q\leqq1,$ there holds
\[
w_{t}-\Delta w=g\leqq|\nabla w|^{q}\leqq|\nabla w|+1.
\]
In case $q=1,$ Lemma \ref{Alaa} applies. Assume that $q<1.$ Let $\varepsilon
,\eta\in(0,1).$ Then $g\leqq C_{\eta}|\nabla w|+\eta.$ with $C_{\eta}%
=\eta^{-q/(1-q)}.$ As in his proof we get by approximation
\begin{align*}
&  \frac{1}{1+\varepsilon}\int_{\Omega}(w^{+})^{1+\varepsilon}%
(t,.)dx+\varepsilon\int_{0}^{t}\int_{\Omega}(w^{+})^{\varepsilon-1}|\nabla
w|^{2}\psi dxdt\\
&  \leqq C_{\eta}\int_{0}^{t}\int_{\Omega}(w^{+})^{\varepsilon}\left\vert
\nabla w\right\vert dxdt+\eta\int_{0}^{t}\int_{\Omega}(w^{+})^{\varepsilon
}dxdt,
\end{align*}
and the second member is finite. Then $\lim_{t\rightarrow0}\int_{\Omega}%
(w^{+})^{1+\varepsilon}(t,.)dx=0,$ hence $\lim_{t\rightarrow0}\int_{\Omega
}w^{+}(t,.)dx=0.$ Let $z=w-\eta t,$ then satisfies $z\in C((0,T);L^{1}\left(
\Omega\right)  )\cap L^{1}((0,T);W^{1,1}\left(  \Omega\right)  )$ and
$z_{t}-\Delta z=g-\eta\leqq C_{\eta}|\nabla z|$ in $\mathcal{D}^{\prime
}\left(  Q_{\Omega,T}\right)  .$ Then $z^{+}\in C((0,T);L^{1}\left(
\Omega\right)  )\cap L^{1}((0,T);W_{0}^{1,1}\left(  \Omega\right)  )$ and from
\cite[Lemma 3.2]{BaPi}, $z_{t}^{+}-\Delta z^{+}\leqq C_{\eta}|\nabla(z^{+})|.$
And $\lim_{t\rightarrow0}z^{+}(t)=0$ weakly in $\mathcal{M}_{b}(\Omega),$
since $z^{+}\leqq w^{+}.$ Then $z^{+}=0$ from Lemma \ref{Alaa} applied with
$A=C_{\varepsilon}.$ Thus $w\leqq\eta t$; as $\eta\rightarrow0,$ we obtain
$w\leqq0.$\medskip
\end{proof}

\begin{remark}
We can give an alternative proof of uniqueness, using regularity: let $u,v$ be
two solutions with initial data $u_{0}$, and $w=u-v,$ thus $w$ satisfies%
\begin{equation}
\left\{
\begin{array}
[c]{l}%
w_{t}-\Delta w=g:=|\nabla u|^{q}-|\nabla v|^{q},\quad\text{in}\hspace
{0.05in}Q_{\Omega,T},\\
w=0,\quad\text{on}\hspace{0.05in}\partial\Omega\times(0,T),\\
\lim_{t\rightarrow0}w(.,t)=0,\text{ weakly in }\mathcal{M}_{b}(\Omega).
\end{array}
\right.  \label{pbg}%
\end{equation}
Since $q\leqq1,$ there holds $\left\vert g\right\vert \leqq|\nabla w|^{q}.$ As
in Theorem \ref{gul2}, we choose $k\in(1,q_{\ast})$, thus $|\nabla w|\in
L^{k}\left(  Q_{\Omega,\tau}\right)  .$ From the uniqueness of the solution
$w$ due to \cite[Lemma 3.4]{BaPi}, we deduce that $w\in\mathcal{W}%
^{2,1,k}(Q_{\Omega,\tau}),$ for any $\tau\in\left(  0,T\right)  ,$ from \cite[
theorem IV.$9.1$]{LSU}. By induction we deduce that $w\in C^{0}\left(
\overline{\Omega}\times\left[  0,T\right)  \right)  \cap C^{2+\gamma
,1+\gamma/2}(Q_{\Omega,T}).$ Then $w=0$ from the classical maximum principle.
\end{remark}

Next we prove the trace result of Theorem \ref{infun}:\medskip

\begin{proof}
[First proof of Theorem \ref{infun}]From Theorem \ref{gul2}, $u\in
C_{loc}^{2,1}(Q_{\Omega,T}).$ And $1+u$ is also a solution of (\ref{un}). We
can set $1+u=v^{\alpha},$ with $\alpha>1,$ in particular $v\geqq1.$ Then we
obtain an equivalent equation for $v:$%
\[
v_{t}-\Delta v=H:=\left(  \alpha-1\right)  \frac{\left\vert \nabla
v\right\vert ^{2}}{v}-\alpha^{q-1}\frac{\left\vert \nabla v\right\vert ^{q}%
}{v^{(\alpha-1)(1-q)}}.
\]
From the Young inequality, setting $C=(\left(  \alpha-1\right)  /2)^{(q-2)/q}%
,$ there holds, since $v\geqq1,$
\[
\frac{\left\vert \nabla v\right\vert ^{q}}{v^{(\alpha-1)(1-q)}}\leqq
\frac{\alpha-1}{2}\frac{\left\vert \nabla v\right\vert ^{2}}{v}+Cv^{1-\frac
{2(1-q)}{2-q}\alpha}\leqq\frac{\alpha-1}{2}\frac{\left\vert \nabla
v\right\vert ^{2}}{v}+Cv.
\]
Hence $w=e^{Ct}v$ satisfies
\[
w_{t}-\Delta w=G:=e^{Ct}(H+Cv)\geqq\frac{\alpha-1}{2}\frac{\left\vert \nabla
w\right\vert ^{2}}{w}.
\]
Then $w$ is supercaloric, and nonnegative, and $G\in L_{loc}^{1}(Q_{T}).$ From
Lemma \ref{phi}, $w$ admits a trace in $\mathcal{M}\left(  \Omega\right)  $,
and then $w\in L_{loc}^{\infty}{(}\left[  0,T\right)  {;L_{loc}^{1}(}%
\Omega)),$ and $G\in L_{loc}^{1}([0,T);L_{loc}^{1}(\Omega)).$ As a
consequence, $v\in L_{loc}^{\infty}{(}\left[  0,T\right)  {;L_{loc}^{1}%
(}\Omega))$ and $\left\vert \nabla v\right\vert ^{2}/v$ $\in L_{loc}%
^{1}([0,T);L_{loc}^{1}(\Omega)).$\medskip

Next we show that moreover $u$ itself admits a trace measure. For any
$0<s<t<T,$ from the H\"{o}lder inequality,
\begin{equation}
\alpha^{-q}\int_{s}^{t}\int_{\omega}\left\vert \nabla u\right\vert
^{q}dxdt=\int_{s}^{t}\int_{\Omega}v^{(\alpha-1)q}\left\vert \nabla
v\right\vert ^{q}dxdt\leqq\int_{s}^{t}\int_{\omega}\frac{\left\vert \nabla
v\right\vert ^{2}}{v}dxdt+\int_{s}^{t}\int_{\omega}v^{\frac{(2\alpha-1)q}%
{2-q}}dxdt. \label{hocc}%
\end{equation}
First suppose $q<1.$ Choosing $\alpha$ such that moreover $1<\alpha\leqq1/q,$
in order that $(2\alpha-1)q\leqq2-q.$ Since $v\in$ $L_{loc}^{\infty}{(}\left[
0,T\right)  {;L_{loc}^{1}(}\Omega)),$ we have $v\in L^{1}(Q_{T}),$ hence
\[
\alpha^{-q}\int_{s}^{t}\int_{\omega}\left\vert \nabla u\right\vert
^{q}dxdt\leqq\int_{0}^{t}\int_{\omega}\frac{\left\vert \nabla v\right\vert
^{2}}{v}dxdt+\int_{0}^{t}\int_{\omega}(v+1)dxdt,
\]
hence $\left\vert \nabla u\right\vert ^{q}\in L_{loc}^{1}(\Omega\times\left[
0,T\right)  ).$ Then $u$ admits a trace $u_{0}\in\mathcal{M}^{+}(\Omega).$
Next assume $q=1.$ From the H\"{o}lder inequality,%
\[
\int_{\omega}\left\vert \nabla v\right\vert dx\leqq\int_{\omega}%
\frac{\left\vert \nabla v\right\vert ^{2}}{v}dx+\int_{\omega}vdx
\]
hence $\left\vert \nabla v\right\vert \in L_{loc}^{1}([0,T);L_{loc}^{1}%
(\Omega)).$ Let $\xi\in\mathcal{D}(\Omega).$ Setting $v\xi=z,$ $z$ is the
unique solution of the problem in $Q_{\Omega,T}$
\[
\left\{
\begin{array}
[c]{l}%
z_{t}-\Delta z=g:=F\xi+v(-\Delta\psi)-2\nabla v.\nabla\psi,\quad
\text{in}\hspace{0.05in}Q_{\Omega,T},\\
z=0,\quad\text{on}\hspace{0.05in}\partial\Omega\times(0,T),\\
\lim_{t\rightarrow0}z(.,t)=\xi u_{0},\text{ weakly in }\mathcal{M}_{b}%
(\Omega),
\end{array}
\right.
\]
where $g\in L^{1}(Q_{\Omega,T}).$ From Theorem \ref{bapi}, for any
$k\in\left[  1,q_{\ast}\right)  ,$ and for any $0<s<\tau<T,$ and any domain
$\omega\subset\subset\Omega,$
\[
\left\Vert z\right\Vert _{L^{k}(Q_{\omega,s,\tau})}\leqq C(\left\Vert
F\xi\right\Vert _{L^{1}(Q_{\omega,s,\tau})}+\left\Vert z(s,.)\right\Vert
_{L^{1}(\omega)})\leqq C(\left\Vert F\xi\right\Vert _{L^{1}(Q_{\omega,\tau}%
)}+\left\Vert v\right\Vert _{L^{\infty}{(}(0,\tau){;L^{1}(}\omega))})
\]
Then $z\in L_{loc}^{k}([0,T);L^{k}(\Omega)).$ We can choose $\alpha$ such that
$1<\alpha<1+q_{\ast}/2,$ and take $k=2\alpha-1.$ From (\ref{hocc}) we deduce
that $\left\vert \nabla u\right\vert \in L_{loc}^{1}(\Omega\times\left[
0,T\right)  ),$ and conclude again that $u$ admits a trace $u_{0}%
\in\mathcal{M}^{+}(\Omega).\medskip$
\end{proof}

Finally we give an alternative proof by using comparison with solutions with
initial Dirac mass, inspired of \cite{ASaVe}. We first extend Proposition
\ref{cpro} to the case $q\leqq1$ when $\Omega$ is bounded:

\begin{lemma}
\label{cpru}Let $0<q\leqq1$, $\Omega$ bounded, and $u_{0,n},u_{0}%
\in\mathcal{M}_{b}^{+}(\Omega)$ such that $u_{0,n}$ converge to $u_{0}$ weakly
in $\mathcal{M}_{b}(\Omega).$ Let $u_{n},u$ be the unique nonnegative
solutions of $(D_{\Omega,T})$ with initial data $u_{0,n},u_{0}.$ Then $u_{n}$
converges to $u$ in $C_{loc}^{2,1}(Q_{\Omega,T})\cap C^{1,0}\left(
\overline{\Omega}\times\left(  0,T\right)  \right)  $.
\end{lemma}

\begin{proof}
We still have (\ref{auto}) and lim$_{s\rightarrow0}\int_{\Omega}%
u_{n}(s,.)dx=\int_{\Omega}du_{0,n},$ thus $\int_{\Omega}u_{n}(t,.)dx\leqq
\int_{\Omega}du_{0,n},$ and $\lim_{n\rightarrow\infty}\int_{\Omega}%
du_{0,n}=\int_{\Omega}du_{0},$ thus $(u_{n})$ is bounded in $L^{\infty
}((0,T);L^{1}(\Omega)).$ From Theorem \ref{gul2}, one can extract a
subsequence converging in $C_{loc}^{2,1}(Q_{\Omega,T})\cap C^{1,0}\left(
\overline{\Omega}\times\left(  0,T\right)  \right)  $ to a weak solution $w$
of $(D_{\Omega,T})$. And $\left(  u_{n}\right)  $ is bounded in $L^{k}%
((0,T),W_{0}^{1,k}(\Omega)$ for any $k\in\left[  1,q_{\ast}\right)  .$ As in
Proposition \ref{cpro}, for any $\tau\in\left(  0,T\right)  ,$ $($ $\left\vert
\nabla u_{n}\right\vert ^{q})$ is equi-integrable in $Q_{\Omega,\tau},$ and we
conclude that $w=u.\medskip$
\end{proof}

\begin{proof}
[Second proof of Theorem \ref{infun}]We still have $u\in C_{loc}%
^{2,1}(Q_{\Omega,T})$ from Theorem \ref{gul2}. It is enough to show that for
any ball $B(x_{0},\rho)\subset\subset\Omega,$ there exists a measure $m_{\rho
}\in\mathcal{M}(B(x_{0},\rho))$ such that the restriction of $u$ to
$B(x_{0},\rho)$ admits a trace $m_{\rho}\in\mathcal{M}(B(x_{0},\rho)).$
Suppose that it is not true. Then from Proposition \ref{dic} and Remark
\ref{dac}, there exists a ball $B(x_{0},\rho)\subset\subset\Omega$ such that
\[
\lim\sup_{t\rightarrow0}\int_{B(x_{0},\rho)}u(.,t)dx=\infty.
\]
We can assume that $x_{0}=0$ and $\rho=1.$ For any $k>0$, the Dirichlet
problem $(P_{B_{1},T})$ with initial data $k\delta_{0}$ has a unique solution
$u_{k}^{B_{1}}$. Denoting by There exists $t_{1}>0$ such that $\int%
_{B_{2^{-1}}}u(x,t_{1})dx>k$; thus there exists $s_{1,k}>0$ such that
$\int_{B_{2^{-1}}}\min(u(x,t_{1}),s_{1,k})dx=k.$ By induction, there exists a
decreasing sequence $\left(  t_{n}\right)  $ converging to $0,$ and a sequence
$\left(  s_{n,k}\right)  $ such that $\int_{B_{2^{-n}}}\min(u(x,t_{n}%
),s_{n,k})dx=k.$ Denote by $u_{n,k}$ the solution of $(P_{B_{1},T})$ with
initial data $u_{n,k,0}=\chi_{B_{2^{-n}}}\min(u(.,t_{n}),s_{n,k}).$ Then
$u\geqq u_{n,k}$ in $B_{1},$ from Theorem \ref{exun}.  And ($u_{n,k,0})$
converges weakly in $\mathcal{M}_{b}(\Omega)$ to $k\delta_{0}.$ From Lemma
\ref{cpru}, $(u_{n,k})$ converges in $C_{loc}^{2,1}(Q_{B_{1},T})\cap
C^{1,0}\left(  \overline{B_{1}}\times\left(  0,T\right)  \right)  $ to the
solution $u^{k,B_{1}}$ of the problem in $B_{1}$ with initial data
$k\delta_{0},$ Thus $u\geqq u^{k,B_{1}}.$ Now, since $q\leqq1,$ for any $k>1,$
the function $ku^{1,B_{1}}$ is a subsolution of (\ref{un}), since $\left\vert
\nabla(ku^{1,B_{1}})\right\vert ^{q}\leqq k\left\vert \nabla(u^{1,B_{1}%
})\right\vert ^{q}.$ From Lemma \ref{Alaa}, we deduce that $u\geqq
ku^{1,B_{1}}$ for any $k>1.$ Since $u^{1,B_{1}}$ is not identically $0,$ we
get a contradiction as $k\rightarrow\infty.$
\end{proof}

\end{document}